\documentclass[final,onefignum,onetabnum]{siamart190516}

\usepackage{amsmath,amssymb,amsopn}
\usepackage{graphicx}
\usepackage{color}
\usepackage[T1]{fontenc}
\usepackage{booktabs}

\usepackage{algorithm}
\usepackage{algpseudocode} 

\algrenewcommand{\algorithmiccomment}[2][.6\linewidth]{\leavevmode\hfill\makebox[#1][l]{$\triangleright$ #2}}
\algnewcommand{\algorithmicgoto}{\textbf{goto}}%
\algnewcommand{\Goto}[1]{\algorithmicgoto~\ref{#1}}%

\newcommand{\COMMENT}[2][.47\linewidth]{%
  \leavevmode\hfill\makebox[#1][l]{$\triangleright$~#2}}

\usepackage{pgfplots}
\usepackage{pgfplotstable}
\usepgfplotslibrary{groupplots}
\usepgfplotslibrary{colorbrewer}
\usetikzlibrary{pgfplots.groupplots}
\usetikzlibrary{arrows,intersections}

\newsiamthm{thm}{Theorem}
\newsiamthm{cor}{Corollary}
\newsiamthm{lem}{Lemma}
\newsiamremark{rem}{Remark}

\usepackage{doi}

\renewcommand{\Re}{\mathbb{R}}
\newcommand{\I}{N}
\newcommand{\ead}{\dot{\mu}}
\newcommand{\ea}{\mu}
\newcommand{\bet}{\beta}
\newcommand{\cratio}{\vartheta}
\newcommand{\ra}{\rho_1}
\newcommand{\rb}{\rho_2}
\newcommand{\rc}{\rho_3}

\newcommand{\rh}{\rho}
\newcommand{\A}{\psi}
\newcommand{\kap}{\underline{\kappa}(A)}

\newcommand{\M}{\alpha_M}

\newcommand{\m}{M}

\newcommand{\EM}{\|\m\|}

\newcommand{\ir}{$\mathcal{IR}$}
\newcommand{\tg}{$\mathcal{TG}$}
\newcommand{\V}{$\mathcal{V}$}
\newcommand{\fmg}{$\mathcal{FMG}$}

\newcommand{\alp}{\dot{m}_P^+}
\newcommand{\alpnodot}{m_P^+}

\newcommand{\alabar}{{\bar m}_A^+}
\newcommand{\aladot}{{\dot{m}}_A^+}

\newcommand{\cc}{\sigma}
\newcommand{\ewe}{\boldsymbol{\varepsilon}}
\newcommand{\ewed}{\boldsymbol{\dot{\varepsilon}}}
\newcommand{\eweb}{\boldsymbol{\bar{\varepsilon}}}

\newcommand{\rinner}{\rho}
\newcommand{\rtgstar}{\rho_{tg}^*}
\newcommand{\rvstar}{\rho_{v}^*}
\newcommand{\rltg}{\rho_{ltg}}
\newcommand{\rtg}{\rho_{tg}}
\newcommand{\rtghat}{{\hat \rho}_{tg}}
\newcommand{\rv}{\rho_{v}}

\newcommand{\pvar}{\tau}

\newcommand{\pvarb}{{\bar \pvar}}
\newcommand{\pvard}{{\dot{\pvar}}}
\newcommand{\pvarjm}{\pvard_{j - 1}}
\newcommand{\pvarj}{\pvard_{j}}
\newcommand{\pvarone}{\pvard_{1}}
\newcommand{\pir}{\delta_{\rho_{ir}}}
\newcommand{\thresh}{\chi}
\newcommand{\ptg}{\delta_{\rho_{tg}}}
\newcommand{\ptghat}{{\hat \delta}_{\rho_{tg}}}
\newcommand{\ptgjmb}{{\delta}_{\rho_{tg}}(\pvarjm)}
\newcommand{\ptgjb}{{\delta}_{\rho_{tg}}(\pvarj)}
\newcommand{\pltg}{\delta_{\rho_{ltg}}}
\newcommand{\ptgoneb}{{\delta}_{\rho_{tg}}(\pvarone)}
\newcommand{\pv}{\delta_{\rho_{v}}}

\newcommand{\gmm}{\gamma}

\newcommand{\hot}{\varphi}

\newcommand{\emm}{m}
\newcommand{\bigO}{\mathcal{O}}

\definecolor{lo}{HTML}{008000}
\definecolor{med}{HTML}{0000FE}
\definecolor{hi}{HTML}{FB0106}

\DeclareMathOperator{\fl}{fl}

\headers{Error analysis for mixed-precision multigrid}{S. F. McCormick, J. Benzaken, and R. Tamstorf}

\title{Algebraic error analysis for mixed-precision multigrid solvers\thanks{Submitted to the editors June 29, 2020.}}

\author{Stephen F. McCormick\thanks{University of Colorado at Boulder, Boulder, CO
  (\email{stephen.mccormick@colorado.edu}).}
\and Joseph Benzaken\thanks{Walt Disney Animation Studios, Burbank, CA
  (\email{Joseph.Benzaken@disneyanimation.com}, \newline\email{Rasmus.Tamstorf@disneyanimation.com}).}
\and Rasmus Tamstorf\footnotemark[3]}

\ifpdf
\definecolor{darkblue}{rgb}{0.0,0.0,0.3}
\hypersetup{
  pdftitle={Algebraic error analysis for mixed-precision multigrid solvers},
  pdfauthor={S. McCormick, J. Benzaken, and R. Tamstorf}
  linkcolor=darkblue,
  citecolor=darkblue,
  filecolor=black,
  urlcolor=darkblue,
  anchorcolor=darkblue
}
\fi

\begin{document}

\maketitle

\begin{abstract}
 This paper establishes the first theoretical framework for analyzing the rounding-error effects on multigrid methods using mixed-precision iterative-refinement solvers. While motivated by the sparse symmetric positive definite (SPD) matrix equations that arise from discretizing linear elliptic PDEs, the framework is purely algebraic such that it applies to matrices that do not necessarily come from the continuum. Based on the so-called {\em energy} or $A$ norm, which is the natural norm for many problems involving SPD matrices, we provide a normwise forward error analysis, and introduce the notion of progressive precision for multigrid solvers. Each level of the multigrid hierarchy uses three different precisions that each increase with the fineness of the level, but at different rates, thereby ensuring that the bulk of the computation uses the lowest possible precision. The theoretical results developed here in the energy norm differ notably from previous theory based on the Euclidean norm in important ways. In particular, we show that simply rounding an exact result to finite precision causes an error in the energy norm that is proportional to the square root of $\kappa$, the associated matrix condition number. (By contrast, this error is of order $1$ when measured in the Euclidean norm.) Given this observation, we show that the limiting accuracy for both V-cycles and full multigrid is optimal in the sense that it is also proportional to $\kappa^{1/2}$ in energy. Additionally, we show that the loss of convergence rate due to rounding grows in proportion to $\kappa^{1/2}$, but argue that this loss is insignificant in practice. The theory presented here is the first forward error analysis in the energy norm of iterative refinement and the first rounding error analysis of multigrid in general.
\end{abstract}



\section{Introduction}

Most computing systems today are power limited. This is true all the way from battery operated edge devices to the top supercomputers of the world. Current trends in computer architectures therefore favor computation with low precision arithmetic as it allows higher throughput and reduces the amount of data that must be moved through the memory hierarchy. As an example, using single precision allows for roughly four times higher throughput than double precision, \cite{Galal2012}, while the savings due to memory traffic can be even larger depending on how close to the processor the data is stored, \cite{Pedram2017}. Unfortunately, ill-conditioned problems are often intractable in very low precision using standard methods. As a result, much research over the years has  focused on algorithms using mixed-precision computation, with most operations performed in low precision while select operations use higher precision to achieve higher accuracy. This approach was first introduced in \cite{Wilkinson1948}, and later analyzed in \cite{Wilkinson1963} and \cite{Moler1967}. Since then, it has become a well-known technique and, with the introduction of GPUs, much research has focused on combining single and double precision \cite{Langou2006,Buttari2007,Buttari2008,Baboulin2009}. More recently, the addition of hardware support for half-precision (FP16) computations has led to a push toward the use of FP16 \cite{Haidar2017,Haidar2018}. Mixed-precision approaches are also beneficial for minimizing the need for extended precision when dealing with problems that are inherently so ill-conditioned that they are difficult to solve even using today's high-precision arithmetic. Examples of such ill-conditioned problems include high-resolution discretizations of high-order partial differential equations (PDEs). 

As described in \cite{Haidar2017}, many mixed-precision methods fall within the broad class of nested \emph{inner-outer} iterative methods. In this context, {\em outer} refers to an upper-level solver that calls the {\em inner} method to perform essential solver tasks. In this paper, we consider a particular algorithm belonging to this class of methods. Specifically, we use iterative refinement (IR) with three precisions as the outer solver (similar to \cite{Carson2018}) and a \emph{restarted} multigrid V-cycle as the inner solver. As an important contribution, we extend the traditional notion of mixed-precision IR to progressive precision by allowing all precision levels to increase with each new fine level in the multigrid hierarchy. Our goal is to establish a purely {\em algebraic} theory for this method that guarantees convergence when the matrix condition number times unit roundoff is about one at each level of the hierarchy. The estimates confirm that convergence is achieved at nearly the same rate as in infinite precision. The result is accuracy comparable to an algebraic sense of discretization accuracy that is motivated by, but not dependent on, discretized elliptic PDEs. The theory also establishes optimal convergence to this abstract discretization accuracy for progressive-precision full multigrid (FMG). While mixed-precision methods have previously been considered for a variety of multigrid algorithms (e.g., \cite{Strzodka2006,Goeddeke2010,Clark2010}), convergence theory has to the best of our knowledge not been established before.

Our forward rounding error analysis for iterative refinement parallels \cite{Carson2017}, with the crucial difference that we focus on (sparse) symmetric positive definite (SPD) matrices. This  allows us to pose the theory in the energy norm, which is the natural and desired norm for a large class of elliptic PDEs. Physical principles are often posed in terms of minimizing  some physical quantity such as energy, leading to PDEs of this class equipped with a naturally induced energy norm.  

Extending the results of \cite{Carson2017} from the infinity to the energy or $A$ norm is critical because it enables access to standard variational theory for V-cycles and FMG, and the resulting forward analysis enables direct estimation of the rounding-error effects involving $\kappa^{1/2}$ instead of $\kappa$, where $\kappa$ is the matrix condition number. (All norms on a \emph{fixed} Euclidean space are theoretically equivalent, but the constants really matter: those relating the energy and Euclidean norms depend on $\kappa$. As such, the disparity in these constants grows rapidly up through the hierarchy of levels that we consider.) We also exploit matrix sparsity to eliminate the direct dependence of the error bounds on the size of the matrix, thus avoiding an even higher power of the condition number in the estimates. These advantages for the sparse SPD case allow us to sharpen existing theory for general matrices, most notably those in \cite{Carson2017}. 

The results developed here in the energy norm differ notably from previous theory based on the Euclidean norm in other important ways. In particular, we show that simply rounding an exact result to finite precision causes an error in the energy norm that is proportional to $\kappa^{1/2}$. (By contrast, this error is of order $1$ when measured in the Euclidean norm.) Given this observation, the limiting accuracy for both V-cycles and full multigrid proves to be optimal in the sense that it is also proportional to $\kappa^{1/2}$ in energy. The loss of convergence rate from rounding also grows in proportion to $\kappa^{1/2}$, but we argue that this loss is insignificant in practice.


By definition, multigrid coarsening in correction form applies to the residual equation, so it is naturally an iterative-refinement process. Each step of a residual-based relaxation scheme can also be interpreted as iterative refinement. A mixed-precision version of multigrid therefore requires very little change in the algorithm itself, and that change is primarily in the choice of where to invoke higher precision. Our choice here is to apply a simple multigrid algorithm in low precision to the residuals computed in higher precision between cycles, where the cycles we study involve one relaxation sweep on each grid. (We establish an extension to multiple sweeps in the supplemental material.)

We begin in the next section by describing tools used in our theory in the form of basic error-analysis estimates for floating point arithmetic. Section~\ref{sec:ir} reviews the classical iterative-refinement process while detailing the notation, terminology, and certain principles that are used throughout the remaining sections. We develop a general theory in Section~\ref{sec:analysis-ir} that provides error bounds for iterative refinement assuming an abstract energy-convergent inner solver.  Section~\ref{sec:twogrid} then introduces a reformulation of the two-grid method that leverages mixed-precision iterative refinement. The V-cycle follows in Section~\ref{sec:v-cycle}. Theoretical results are presented in Section~\ref{sec:analysis-twogrid} and \ref{sec:analysis-mg} that, together with our iterative refinement estimates, provide a framework for determining convergence of mixed-precision two-grid and V-cycle solvers for specific applications. The notion of progressive precision is introduced in Section~\ref{sec:analysis-mg} as part of the analysis of the V-cycle. Progessive-precision FMG is then introduced in Section~\ref{sec:fmg} and analyzed in Section~\ref{sec:analysis-fmg}. The paper concludes in Section~\ref{sec:conclusion}.

\section{Floating Point Estimates}
\label{sec:fp}
We consider three basic floating point environments: ``standard'' precision with unit roundoff $\ewe$, referred to here as $\ewe$-precision; ``high'' precision with unit roundoff $\eweb$, referred to here as $\eweb$-precision; and ``low'' precision with unit roundoff $\ewed$, referred to here as $\ewed$-precision. Initially, one can think of $\eweb$-precision as being double, and $\ewed$-precision as being half, that of $\ewe$-precision, but we only formally require ${\eweb} \le \ewe \le \ewed$. 

Our analysis is based on existing models of error due to rounding. See \cite{Higham2002} for a thorough discussion of these models and other aspects of rounding error analysis. This section describes only those estimates used in our theory.

Let $\circ$ stand for one of the four basic binary operations $+, -, *, /$ between two scalars $x$ and $y$, and let $\fl(x\circ y)$ stand for the $\ewe$-precision computed value of the exact operation $x\circ y$. The standard model \cite[Eq. (2.4)]{Higham2002} that accounts for rounding error (in $\ewe$-precision for illustration) is then given by
\begin{equation}
\fl(x\circ y) = (x\circ y)(1 + \delta), \quad |\delta| \le \ewe ,
\label{fl1}
\end{equation}
where $\fl()$ denotes the result of the floating point operation. (The definition of $\delta$ with and without subscripts may change meaning at each occurrence in this paper, especially here and from one proof to the next.) An alternative to (\ref{fl1}) is the variant \cite[Eq. (2.5)]{Higham2002}, which we write as
\begin{equation}
\fl(x\circ y) = x\circ y + \delta \fl(x\circ y), \quad |\delta| \le \ewe.
\label{fl3}
\end{equation}
Assuming now that $A \in \Re^{n \times n}$ and $x \in \Re^n$, then the dot product model \cite[Eq. (3.4)]{Higham2002} implies the following matrix-vector product estimate in $\ewe$-precision:
\begin{equation}
\fl(Ax) = Ax + \delta, \quad |\delta| \le \frac{n\ewe}{1 - n \ewe} |A|\cdot |x|,
\label{fl5a}
\end{equation}
where $|z|$ denotes the vector of the absolute values of $z \in \Re^n$ (similarly for matrices) and  relations between vectors and matrices are defined componentwise. With $\| \cdot \|$ denoting the Euclidean norm for a vector and its induced matrix norm (together with the Euclidean inner product $\langle \cdot, \cdot \rangle$), we also frequently use the fact that $\||z|\| = \|z\|$ for any vector $z$ (although this is not generally true for matrices).
If $b \in \Re^n$, then an estimate for computing the residual $Ax - b$ in $\ewe$-precision is given by
\begin{equation}
\fl(Ax - b) = Ax - b + \delta, \quad |\delta| \le \frac{(n + 1)\ewe}{1 - (n + 1)\ewe} (|b| + |A|\cdot |x|).
\label{fl5b}
\end{equation}
Finally, computing the residual in $\eweb$-precision and rounding it to $\ewe$-precision results in the following bound \cite[Eq. (2.3)]{Carson2017}:
\begin{equation}
\fl(Ax - b) = Ax - b + \delta, \quad |\delta| \le \ewe|Ax - b| + (1 + \ewe) \frac{(n + 1){\eweb}}{1 - (n + 1){\eweb}} (|b| + |A|\cdot |x|).
\label{fl6}
\end{equation}

We have used $\fl$ here to indicate computed quantities, but it would be too cumbersome to continue with this notation in the analysis that follows. Instead of $\fl$ or any other special way to denote computed quantities, we add $\delta$'s to exact expressions, with various subscripts and bounds, to denote quantities computed in finite precision. For example, the {\em computed} residual $r$ in (\ref{fl6}) would be written as
\[
r = Ax - b + \delta_r, \quad |\delta_r| \le \ewe |Ax - b| + (1 + \ewe) \frac{n{\eweb}}{1 - n{\eweb}} (|b| + |A|\cdot |x|)
\]
and the {\em computed} solution of $Ax = b$ would be written as $A^{-1}b + \delta$.

Following the usual convention in rounding-error analyses, we assume throughout that the system matrix $A$ and right-hand side (RHS) $b$ for our target problem (see (\ref{axb}) below) are exact. To rein in the complexity of our convergence bounds, we take this assumption further by assuming exactness of all of the multigrid components: the intergrid transfer, the system matrix, and the RHS on all levels. While this may be a reasonable assumption in certain ideal cases (e.g., simple discretizations of Poisson's equation), these quantities are only approximate in most applications, such as when they are constructed via finite elements using quadrature for a weak form applied to individual basis elements. Even if these approximations are very accurate, $A$ and $b$ must ultimately be stored in finite precision. While it may be argued that errors in the target problem are not at issue here since our focus is on the discrete problem we are given, accuracy of the coarse components within the multigrid solvers is important. This is particularly true for FMG since its goal is to deliver discretization-order accuracy, which makes it the most error-sensitive multigrid scheme. See \cite{Benzaken2020} for an analysis of the effects of finite precision on the coarse FMG components.

\section{Iterative Refinement}
\label{sec:ir}

Consider the $n \times n$ matrix equation
\begin{equation}
Ax = b,
\label{axb}
\end{equation}
where $n$ is a positive integer, $A \in \Re^{n \times n}$ is SPD, $x \in \Re^n$ is unknown, and $b \in \Re^n$ is given. Assume that $A$ is sparse, with at most $m_A$ nonzeros per row. The first mixed-precision approach we analyze in this paper uses IR as the outer loop and an approximate linear solver (to be specified later) as the inner loop. The pseudocode {\ir } is given in Algorithm \ref{alg-ir} below. As noted by \cite[Sec. 2.5]{Demmel1997}, {\ir } is effectively Newton's method applied to the function $f(x)=Ax-b$, where $f(x)$ is considered as a nonlinear function due to rounding errors. It can also be interpreted in many other ways, including preconditioned Richardson iteration \cite{Buttari2008}, defect correction \cite{Chatelin1983}, and double discretization \cite{BrandtLivne2001}.

The floating-point operations in {\ir } here use all three precisions. The full residual $r= Ax - b$ between successive calls to the inner solver is evaluated in $\eweb$-precision (red font), while the inner solver itself uses $\ewed$-precision (green font). All other operations, and most notably the update step, are evaluated using $\ewe$-precision (blue font). Note that step 2 of {\ir } is in red and blue font because the residual is computed in $\eweb$-precision but rounded to $\ewe$-precision. A critical feature of {\ir } is that the inner loop is applied not to the full equation $Ax = b$ for the full approximation $x$, but rather to the {\ir } {\em residual equation} $Ay = r$ for the correction $y$.

\noindent
\begin{center}
\begin{minipage}{0.8\linewidth}
\begin{algorithm}[H]
\footnotesize
\caption{$\texttt{Iterative Refinement (\ir)}$}
\label{alg-ir}
\begin{algorithmic}[1]
      \Require A, b, x initial guess, tol $> 0$ convergence tolerance.
      \State \textcolor{med}{$r \leftarrow$}\textcolor{hi}{$ Ax - b$}\Comment{\textcolor{hi}{Compute {\ir } Residual} and \textcolor{med}{Round}}\label{ir-begin}
      \If{\textcolor{med}{$\|r\| <$ tol}}
      \State\Return $x$\Comment{\textcolor{med}{Return Solution of $Ax = b$}}
      \EndIf
      \State \textcolor{med}{$y \leftarrow$} \textcolor{lo}{$\texttt{InnerSolve}(A, r)$}\Comment{\textcolor{lo}{Compute Approximate Solution of $Ay = r$}}
      \State \textcolor{med}{$x \leftarrow x - y$}\Comment{\textcolor{med}{Update Approximate Solution of $Ax=b$}}
      \State\Goto{ir-begin}
\end{algorithmic}
\end{algorithm}
\end{minipage}
\end{center}
\vspace{1em}

For simplicity, {\ir } here uses the residual error in its stopping criterion. When the matrix is highly ill-conditioned, the residual can be a poor indicator of convergence because most solvers tend to produce errors whose components are predominantly in the lower spectrum of the matrix and are thus hidden in the residual. Fortunately, in practice, the residual is usually a good indicator for multigrid: while a few relaxation sweeps generally lead to the lower end of the spectrum dominating the error, coarsening tends to reduce those components so that the upper end is exposed; taken together, relaxation and coarsening tend to produce errors that are balanced across the spectrum. Nevertheless, the stopping criterion for mixed-precision algorithms should be chosen carefully as \cite{Baboulin2009} recommends. This issue is discussed further in \cite{Arioli1989} and \cite{Demmel2006}, and we develop a stopping criterion carefully for full multigrid in the related paper, \cite{Benzaken2020}, based on estimates of the differences between approximations on successive grid levels.

\section{Convergence Theory: Iterative Refinement}
\label{sec:analysis-ir}

The statement and proof of our first theorem follow a course similar to that in \cite{Carson2017}. However, instead of the Euclidean norm $\| \cdot \|$, our rounding-error estimates are in terms of the discrete {\em energy} norm $\| \cdot \|_A$ defined by $\|x\|_A = \|A^{\frac{1}{2}}x\|, x \in \Re^n$. Our bounds are also tighter in that they exploit the sparsity of $A$. We state this theorem in terms of a general iterative solver satisfying a convergence bound with generic factor $\rh$. Later, we consider specific multigrid solvers for the inner loop.

Several parameters appear below in our convergence estimates, such as the condition number of $A$ denoted by $\kappa(A)=\|A\| \cdot \|A^{-1}\|$. For convenience, let
\[
\A = \||A|\|, \, \kap = \A \|A^{-1}\|, \textrm{  and  } \alabar = \frac{m_A + 1}{1 - (m_A + 1){\eweb}}.
\]
All bounds obtained herein are expressed in terms of the parameters $\pvard = \kappa^\frac{1}{2}(A)\ewed$, $\pvar = \kappa^\frac{1}{2}(A)\ewe$, and $\pvarb = \kappa(A){\eweb}$ that represent fundamental scales. Quantities representing rounding errors in convergence factors (e.g., $\pir$ and $\ptg$) and {\ir's } limiting accuracy (i.e., $\thresh$) are functions of these  parameters, but this is suppressed in the notation when that dependence is clear. To simplify the bounds below, let
\begin{equation}
\gmm = \frac{\kappa^{\frac{1}{2}}(A) + \kap}{\kappa(A)},
\label{gmm}
\end{equation}
which is of order $1$ for discrete Laplacians and many other discrete elliptic PDEs.

Throughout the paper, we use the notation of the pseudo algorithms we present, with superscripts used occasionally within the iterations to keep track of various quantities. Thus, $x^{(i)}$ denotes the $i^\textrm{th}$ {\ir } iterate and $x^{(0)}$ the {\ir } initial guess. 

\begin{thm}{\em \ir.}
Let $x^{(i)}$ be the iterate at the start of the $i^\textrm{th}$ cycle of {\ir } and $r = Ax^{(i)} - b$ its residual computed in ${\eweb}$-precision and rounded to $\ewe$-precision. Suppose that a $\rinner < 1$ exists such that, for any $r \in \Re^n$, the solver used in the inner loop of Algorithm~\ref{alg-ir} (line 5) is guaranteed to compute a correction $y$ that satisfies
\begin{equation}
\|y - A^{-1} r\|_A \le \rinner \|A^{-1} r\|_A .
\label{rh}
\end{equation}
Then $x^{(i+1)}$ approximates the solution $A^{-1} b$ of (\ref{axb}) with the {\em relative} error bound
\begin{equation}
\frac{\|x^{(i+1)} - A^{-1}b\|_A}{\|A^{-1}b\|_A} \le \rho_{ir} \frac{\|x^{(i)} - A^{-1}b\|_A }{\|A^{-1}b\|_A} + \thresh,
\label{relbnd}
\end{equation}
where
\begin{equation}
\rho_{ir} = \rinner + \pir, \quad \pir = \frac{(1 + 2\rinner) \pvar + \gmm(1 + \rinner)(1 + \ewe)\alabar \pvarb}{1 - \pvar},
\label{ec}
\end{equation}
and
\begin{equation}
\thresh = \frac{\pvar
+ \gmm(1 + \rinner)(1 + \ewe) \alabar\pvarb}{1 - \pvar}.
\label{chi}
\end{equation}
If $\rinner + \pir < 1$, then the error after $\I \ge 1$ {\ir } cycles starting with initial guess $x^{(0)}$ satisfies
\begin{equation}
\frac{\|x^{(\I)} - A^{-1}b\|_A}{\|A^{-1}b\|_A} \le (\rinner + \pir)^\I \frac{\|x^{(0)} - A^{-1}b\|_A}{\|A^{-1}b\|_A} + \frac{\thresh }{1 -(\rinner + \pir)} .
\label{Kbnd}
\end{equation}

\label{theorem2}
\end{thm}

\begin{proof}
Using (\ref{fl6}) allows us to write the {\ir } residual computed in $\eweb$-precision and rounded to $\ewe$-precision as
\begin{equation}
r = \underbrace{Ax^{(i)} - b}_\textrm{exact residual} + \underbrace{\delta_1}_\textrm{$\ewe$-${\eweb}$ error},
\quad |\delta_1| \le \ewe |Ax^{(i)} - b| + (1 + \ewe)\alabar {\eweb} \left(|b| + |A| \cdot |x^{(i)}|\right).
\label{r0}
\end{equation}
Note that
\begin{align*}
\left\|A^{-\frac{1}{2}}\left(|b| + |A| \cdot |x^{(i)}|\right)\right\| &\le \|A^{-\frac{1}{2}}\| \left(\|AA^{-1}b\| + \A \|x^{(i)}\| \right)\\
&\le \|A^{-\frac{1}{2}}\| \cdot \|A^{\frac{1}{2}}\| \cdot \|A^{-1}b\|_A + \A \|A^{-1}\| \cdot \|x^{(i)}\|_A \\
&= \kappa^{\frac{1}{2}}(A) \|A^{-1}b\|_A + \kap \|x^{(i)}\|_A \\
&\le (\kappa^{\frac{1}{2}}(A) + \kap)\|A^{-1}b\|_A + \kap \|x^{(i)} - A^{-1}b\|_A .
\end{align*}
Hence,
\begin{align}
\|A^{-1}\delta_1\|_A
&= \|A^{-\frac{1}{2}}\delta_1\| \nonumber \\
&\le \ewe\|A^{-\frac{1}{2}}\|\cdot \|Ax^{(i)} - b\|  + (1 + \ewe)\alabar {\eweb}\left\|A^{-\frac{1}{2}}\left(|b| + |A| \cdot |x^{(i)}|\right)\right\| \nonumber \\
&\le \ewe\kappa^\frac{1}{2}(A)\|x^{(i)} - A^{-1}b\|_A  \nonumber \\
&\quad +(1 + \ewe)\alabar {\eweb} \left((\kappa^{\frac{1}{2}} (A) + \kap)\|A^{-1}b\|_A + \kap \|x^{(i)} - A^{-1}b\|_A \right).
\label{del}
\end{align}
By (\ref{rh}) and (\ref{r0}), we thus have that
\begin{align}
\|y &- A^{-1}r\|_A \nonumber \\
&\le \rinner \|A^{-1}r\|_A \nonumber \\
&= \rinner \|A^{-1}\left(Ax^{(i)} - b + \delta_1\right)\|_A \nonumber \\
&\le \rinner \left[\|x^{(i)} - A^{-1}b\|_A + \|A^{-1}\delta_1\|_A\right] \nonumber \\
&\le \rinner \left[(1 + \ewe\kappa^\frac{1}{2}(A))\|x^{(i)} - A^{-1}b\|_A\right. \nonumber \\
&\quad \left.+ (1 + \ewe)\alabar {\eweb}\left((\kappa^{\frac{1}{2}}(A) + \kap)\|A^{-1}b\|_A + \kap \|x^{(i)} - A^{-1}b\|_A \right)\right] .
\label{y1}
\end{align}
Using (\ref{fl3}) allows us to write the {\ir } update computed in $\ewe$-precision as
\[
x^{(i + 1)} = \underbrace{x^{(i)} - y}_\textrm{exact update} + \underbrace{\delta_2}_\textrm{$\ewe$ error}, \quad |\delta_2| \le \ewe |x^{(i + 1)}|.
\]
Using (\ref{del}) and (\ref{y1}) thus leads to
\begin{align}
\|x^{(i + 1)}& - A^{-1}b\|_A = \|x^{(i)} - A^{-1}b - A^{-1}r - (y - A^{-1}r) + \delta_2\|_A \nonumber \\
&\le  \|A^{-1} \delta_1 + (y - A^{-1}r)\|_A + \|\delta_2\|_A \nonumber \\
&\le \|A^{-1} \delta_1\|_A + \|y - A^{-1}r\|_A + \ewe\|A^{\frac{1}{2}}\|\cdot\|x^{(i + 1)}\| \nonumber \\
&\le \left(\rinner + (1 + \rinner)\ewe\kappa^\frac{1}{2}(A))\right)\|x^{(i)} - A^{-1}b\|_A \nonumber \\
&\quad + (1 + \rinner)(1 + \ewe)\alabar {\eweb}\left((\kappa^\frac{1}{2}(A) + \kap)\|A^{-1}b\|_A + \kap \|x^{(i)} - A^{-1}b\|_A \right) \nonumber \\
&\quad + \ewe \kappa^\frac{1}{2}(A)(\|x^{(i + 1)} - A^{-1}b\|_A + \|A^{-1}b\|_A) \nonumber\\
&= \ewe \kappa^\frac{1}{2}(A)\|x^{(i + 1)} - A^{-1}b\|_A + (\rinner + \delta_3) \|x^{(i)} - A^{-1}b\|_A \nonumber \\
&\quad+ \left(\ewe\kappa^\frac{1}{2}(A) + (1 + \rinner)(1 + \ewe)\alabar {\eweb}(\kappa^\frac{1}{2}(A) + \kap)\right) \|A^{-1}b\|_A ,
\label{err1}
\end{align}
where
\[
\delta_3 = (1 + \rinner)\left( \ewe\kappa^\frac{1}{2}(A) + (1 + \ewe)\alabar {\eweb} \kap\right).
\]
Bound (\ref{relbnd}) now follows by subtracting $\ewe \kappa^\frac{1}{2}(A)\|x^{(i + 1)} - A^{-1}b\|_A$, dividing by $(1 - \ewe \kappa^\frac{1}{2}(A))\|A^{-1}b\|_A$ on both sides of (\ref{err1}), appealing to (\ref{gmm}), and noting that
\[
\frac{\rinner + \delta_3}{1 - \ewe \kappa^\frac{1}{2}(A)} = \rinner + \frac{\delta_3 + \rinner \ewe \kappa^\frac{1}{2}(A)}{1 - \ewe \kappa^\frac{1}{2}(A)} = \rinner + \pir.
\]
Bound (\ref{Kbnd}) follows by tracing (\ref{relbnd}) back to the initial error and noting that
\[
\sum_{i = 0}^{\I - 1}(\rinner + \pir)^i \le \frac{1}{1 -(\rinner + \pir)}.
\]
\end{proof}

\begin{rem}{\em Limiting Accuracy.}
 Theorem \ref{theorem2} establishes convergence in energy of {\ir } if $\pir < 1 - \rh$, but only until the relative error reaches $\thresh$, the limiting accuracy. If $\rh \ll 0.9$ for example, then we just want $\pir < 0.1$. For sufficiently small $\eweb$, both $\thresh$ and $\pir$ are of order $\pvar = \kappa^\frac{1}{2} (A) \ewe$, so $\thresh$ tends to be the limiting factor long before the loss in $\pir$ becomes an issue. It should also be noted that the inner solver \emph{only} impacts the analysis by way of its convergence factor. Its precision thus has little effect as long as $\rinner \ll 1$. It is also important to note that this performance of {\ir } is optimal in that it reaches an accuracy comparable to that of the finite element solution that has simply been rounded to $\ewe$-precision. That is, simply rounding the exact finite element solution of the PDE induces an error of the same order as the limiting accuracy. The mechanism at play here is the discontinuity of $\fl$ that allows an arbitrarily small-energy perturbation of a function to induce a large-energy perturbation in the last significant bit. As a simple illustration, suppose that (\ref{axb}) is derived from applying Rayleigh-Ritz on a uniform grid $h$ to the {\em no-flow} 1D Poisson-reaction equation  $-u'' + u = f, \, 0 \le x \le 1, \, u'(0)=u'(1)=0$.  Suppose that the exact solution is $x = 1 + y$, where $y$ is a vector with very small energy: $\|y\|_A < \ewe$. Suppose also that $y$ oscillates from positive to negative values from one grid point to the next.  Letting $1 + z$ denote $x$ {\em rounded towards zero} in $\ewe$-precision, then $v$ must oscillate between $-\ewe$ and $0$, thus yielding a relative energy error of
 \[
 \frac{\|1 + z - 1 - y\|_A}{\|1 + z\|_A} 
 \approx \|z\|_A
\approx \sqrt{2} h^{-1}\ewe 
 = \bigO(\kappa^\frac{1}{2}(A) \ewe).
 \]
This is the same level of error that {\ir } achieves at its limiting accuracy.
\end{rem}

\section{Two-grid solver}
\label{sec:twogrid}

Analysis of the full multigrid method proceeds in several stages. We begin here by writing the two-grid solver in correction form applied to the {\ir } residual equation $Ay = r$. (See \cite{tutorial} for an introduction to multigrid methods and principles.)

To understand multigrid solvers and the development that follows, it is important to be clear about terminology. The term {\em residual} could be confusing if the equation it references is not fully understood. Even for standard multigrid solvers based on the correction scheme, care is needed in using the term {\em residual} in reference to an equation: a grid in the middle of the hierarchy gets its equation as an approximation to the finer-grid residual equation, but it too must pass the residual of its approximate residual equation to the next-coarser grid. While such ambiguity is usually avoided in the literature, it becomes more crucial here because we need to compute residuals related to residual equations even on the fine grid. To avoid this potential confusion, we use the terminology {\tg } {\em residual} for those computed within a two-grid cycle, {\V } {\em residual} for those computed within a V-cycle, and {\ir } {\em residual} for those formed from the full approximation $x$ computed between cycles, and similarly for the equations to which they refer.

The two basic components of any multigrid algorithm are {\em relaxation} and {\em prolongation}. We first assume that relaxation is given by the convergent stationary linear iteration $x \leftarrow x - \m (Ax - b)$, where $M \in \Re^{n \times n}$ is a nonsingular matrix. Matrix $M$ is meant to be an easily computed approximation to $A^{-1}$, such as the inverse of the diagonal or lower-triangular part of $A$. (While it is standard in the matrix splitting literature to use $M^{-1}$, we use $M$ here for simplicity.) Next, assume that prolongation (i.e., interpolation) is given by a matrix $P \in \Re^{n \times n_c}$ defined in terms of a coarse level of $n_c < n$ variables. Note that we use subscript $c$ here to signify a coarse-grid quantity with no subscript for fine-grid quantities. We use this convention when there is no risk of ambiguity, but for more than two grids, subscripts involving level $j$ become necessary.

All computations in the pseudocode {\tg } in Algorithm \ref{alg-tg} below are performed in low $\ewed$-precision (green font), except for the exact, infinite-precision coarse-grid solve (black font). Accordingly, since the input RHS may be in higher precision, the cycle is initialized with a rounding step. The solver then proceeds by relaxing on the initial guess to the solution of the {\ir } residual equation and then improving the result by a coarse-grid correction based on prolongation. 

\noindent
\begin{center}
\begin{minipage}{.8\linewidth}
\begin{algorithm}[H]
\footnotesize
\caption{$\texttt{Two-Grid (\tg) Correction Scheme}$}
\label{alg-tg}
\begin{algorithmic}[1]
\Require A, r, P, M.
\State $\textcolor{lo}{r \leftarrow} r$ \Comment{\textcolor{lo}{Round RHS} and Initialize {\tg }}
\State \textcolor{lo}{$y \leftarrow  \m r$}\Comment{\textcolor{lo}{Relax on Current Approximation}}
\State \textcolor{lo}{$r_{\textrm{tg}} \leftarrow Ay - r$}\Comment{\textcolor{lo}{Evaluate {\tg } Residual}}
\State \textcolor{lo}{$b_c \leftarrow P^tr_{\textrm{tg}}$}\Comment{\textcolor{lo}{Restrict {\tg } Residual to Coarse-Grid}}
\State $d_c \leftarrow B_c(P^tAP)^{-1}b_c$\Comment{Solve Coarse-Grid Equation}
\State \textcolor{lo}{$d \leftarrow Pd_c$}\Comment{\textcolor{lo}{Interpolate Correction to Fine Grid}}
\State \textcolor{lo}{$y \leftarrow y - d$}\Comment{\textcolor{lo}{Update Approximate Solution of $Ay = r$}}
\State\Return $y$\Comment{Return Approximate Solution of $Ay = r$}
\end{algorithmic}
\end{algorithm}
\end{minipage}
\end{center}
\vspace{1em}

{\tg } incorporates $B_c \in \Re^{n_c \times n_c}$ in the coarse-grid solution process to allow our theory to apply to the case of approximate solution of the coarse-grid equation (as opposed to exact inversion of $P^tAP$ assumed for the standard two-grid case). Note that if $B_c = I_c$, the coarse-grid identity matrix, then {\tg } reverts to a standard two-grid scheme that uses an exact solve on the coarse grid. The essential assumption that we make about {\tg } when $B_c = I_c$ is that it converges in infinite precision in energy, that is, its error propagation matrix $TG$ is bounded according to
\begin{equation}
\|T G\|_A < 1,
\label{tgb}
\end{equation}
where $G = I - \m A$ is the error propagation matrix for relaxation and $T = I - P (P^tAP)^{-1}P^tA$ is the error propagation matrix for coarse-grid correction. Vectors in the ranges of $T$ and $P$ are called {\em algebraically oscillatory} and {\em algebraically smooth}, respectively, because they tend to correspond to the respective geometrically oscillatory and geometrically smooth vectors targeted by relaxation and coarse-grid correction. See \cite{tutorial}. For the case of general $B_c$, we assume that $\|B_c - I_c\|_{A_c} < 1$, where $A_c = P^tAP$ and $I_c$ is the coarse-grid identity, and we let $\rtgstar$ denote a bound on the norm of the resulting error propagation matrix:
\begin{equation}
\|(I - PB_c (P^tAP)^{-1}P^tA)G\|_A \le \rtgstar,
\label{gamma}
\end{equation}
where the subscript $tg$ signifies two-grid and the superscript asterisk indicates that it is the infinite-precision factor. To see that bound (\ref{tgb}) allows us to choose $\rtgstar < 1$, note that for any fine-grid vector $y$ with $\|y\|_A = 1$, since $I - T$ and $T$ are energy-orthogonal projections onto the range of $P$ and its energy-orthogonal complement, respectively (c.f. \cite{tutorial}), we have that
\begin{align*}
& \|(I - PB_c (P^tAP)^{-1}P^tA)Gy\|_A^2  \\
&\quad = \|TGy\|_A^2 + \|P(B_c - I_c) (P^tAP)^{-1}P^tA)Gy\|_A^2  \\
&\quad \le \|TGy\|_A^2 + \|B_c - I_c\|_{A_c}^2  \|(P^tAP)^{-1}P^tA)Gy\|_{A_c}^2  \\
&\quad = \|TGy\|_A^2 + \|B_c - I_c\|_{A_c}^2  \|(I - T)Gy\|_A^2  \\
&\quad = (1 - \|B_c - I_c\|_{A_c}^2) \|TGy\|_A^2 + \|B_c - I_c\|_{A_c}^2 ( \|TGy\|_A^2 +  \|(I - T)Gy\|_A^2)  \\
&\quad \le (1 - \|B_c - I_c\|_{A_c}^2) \|TG\|_A^2 +  \|B_c - I_c\|_{A_c}^2 < 1.
\end{align*}
Until we treat the multilevel solver, it is probably best for the reader to keep in mind that choosing $B_c = I_c$ reverts to the standard two-grid case.

Assume further that relaxation converges monotonically in energy in the sense that $\|G\|_A < 1$ and that $P$ has full rank with at most $m_P$ nonzero entries per row or column, all of which are positive. We have in mind that $\|G\|_A$ may be very close to $1$ so that relaxation is a poor solver by itself, but that $\rtgstar$ and, therefore, $\|B_c - I_c\|_{A_c}$ and $\|TG\|_A$ are enough less than $1$ that coarse-grid correction effectively eliminates algebraic errors that relaxation cannot properly attenuate. However, these conditions are neither specified nor required in what follows.

The objective of {\tg } is to approximate the exact solution, $A^{-1}r$, of the {\ir } residual equation
\begin{equation}
Ay = r,
\label{ayr}
\end{equation}
starting from the initial approximation $y = 0$. 
Note our use of $A^{-1}r$ here to denote the exact solution of (\ref{ayr}). As mentioned above, we refer to exact quantities by using expressions like this that characterize them, thereby avoiding the need to introduce additional notation to distinguish between exact and computed quantities. So, while $y = 0$ is of course exact, the {\tg } residuals and iterates are assumed to be computed quantities subject to rounding errors when and as specified in what follows.

\section{Multigrid V-cycle}
\label{sec:v-cycle}

To correspond to {\tg } with one relaxation sweep per cycle, we also analyze the so-called V$(1, 0)$-cycle shown in Algorithm~\ref{alg-mg}, where processing begins on the finest grid and proceeds down through the hierarchy to the coarsest grid, with one relaxation sweep performed on each level along the way. Our focus is on one V-cycle defined recursively by the pseudocode {\V } below on a nested hierarchy of $\ell$ grids from the coarsest $j = 1$ to the finest $j = \ell$. Each level $j \in \{1, 2, \dots, \ell\}$ is equipped with a system matrix $A_j$, with $A_\ell = A$. For each $j \in \{2, 3, \dots , \ell\}$, let $P_j$ be the interpolation matrix that maps from grid $j - 1$ to grid $j$ and let $T_j = I_j -  P_j A_{j - 1}^{-1}P_j^tA_j$, where $I_j$ is the level $j$ identity matrix. For the coarsest grid that involves one relaxation sweep and no further coarsening, we set $P_1 = 0$ so that $T_1$ is just the identity matrix $I_1$. Assume that the following {\em Galerkin condition} \cite{tutorial} is exactly satisfied on all coarse levels:
\begin{equation}
A_{j - 1} = P_jA^jP_j, \quad 2 \le j \le \ell.
\label{galerk}
\end{equation}
See \cite{Benzaken2020} for an analysis of the rounding-error effects on computing (\ref{galerk}).

\noindent
\begin{center}
\begin{minipage}{1.0\linewidth}
\begin{algorithm}[H]
\footnotesize
\caption{$\texttt{V$(1,0)$-Cycle (\V) Correction Scheme}$}
\label{alg-mg}
\begin{algorithmic}[1]
\Require A, r, P, $\ell \ge 1$ {\V } levels.
\State $\textcolor{lo}{r \leftarrow} r$\COMMENT{\textcolor{lo}{Round RHS} and Initialize {\V } }
\State \textcolor{lo}{$y \leftarrow M r$}\COMMENT{\textcolor{lo}{Relax on Current Approximation}}
\If {$\ell > 1$}  \COMMENT{Check for Coarser Grid}
\State \textcolor{lo}{$r_{\textrm{v}} \leftarrow Ay - r$}\COMMENT{\textcolor{lo}{Evaluate {\V } Residual}}
\State \textcolor{lo}{$r_{\ell - 1} \leftarrow P^tr_{\textrm{v}}$}\COMMENT{\textcolor{lo}{Restrict {\V } Residual to Coarse-Grid}}
\State \textcolor{lo}{$d_{\ell - 1} \leftarrow $\V$(A_{\ell - 1}, r_{\ell - 1}, P_{\ell - 1}, \ell - 1)$}\COMMENT{\textcolor{lo}{Compute Correction from Coarser Grids}}
\State \textcolor{lo}{$d \leftarrow Pd_{\ell - 1}$}\COMMENT{\textcolor{lo}{Interpolate Correction to Fine Grid}}
\State \textcolor{lo}{$y \leftarrow y - d$}\COMMENT{\textcolor{lo}{Update Approximate Solution of $Ay = r$}}
\EndIf
\State\Return $y$\COMMENT{Return Approximate Solution of $Ay = r$}
\end{algorithmic}
\end{algorithm}
\end{minipage}
\end{center}
\vspace{1em}

The theory in \cite{MandelMcCormickBank1987} and the references cited therein establish optimal energy convergence in infinite precision of Algorithm~\ref{alg-mg} under fairly general conditions for fully regular elliptic PDEs discretized by standard finite elements. We simply assume this to be the case by supposing that the error propagation matrix $V_j$ for level $j$ is bounded by some $\rvstar \in [0, 1)$ for all $j \in \{1, 2, \dots, \ell \}$. Specifically, noting that the error propagation matrices are defined recursively \cite{McCormick1985}) by
\begin{equation}
V_1 = G_1 \textrm{  and  }V_j = (P_j V_{j - 1} A_{j - 1}^{-1}P_j^tA_j + T_j) G_j, \quad 2 \le j \le \ell ,
\label{vb}
\end{equation}
then we assume that
\begin{equation}
\|V_j\|_{A_j} \le \rvstar, \quad 1 \le j \le \ell .
\label{convv}
\end{equation}

Because $\|V_1\|_{A_1} = \|G_1\|_{A_1}$, a requirement implied by (\ref{convv}) is that $\|G_1\|_{A_1} \le \rvstar$. Such a bound holds in cases that use standard multigrid methods with a sufficiently small  coarse grid applied to a well-posed coarse-grid matrix equation. While this case does not include all potential multigrid applications, it is beyond the scope of the present work to consider situations where full coarsening is difficult or the coarse-grid matrices are ill-conditioned.

The principal aim of this paper is an abstract algebraic theory that applies to both algebraic and geometric multigrid when applied to a large class of PDEs. Accordingly, we have in mind matrices whose condition numbers depend on the mesh size, $h$. (While we do not exclude coarsening in terms of the degree, $p$, of the discretization explicitly, our focus is on coarsening in terms of $h$.) To abstract this $h$-dependence, define the {\em pseudo mesh size} $h_j = \kappa^{-\frac{1}{2\emm}} (A_j), \, 1 \le j \le \ell$, where $\emm$ is a positive integer, and the {\em pseudo mesh coarsening factor} 
\begin{equation}
\theta_j =  \frac{h_{j - 1}}{h_j} , \quad 2 \le j \le \ell .
\label{Theta}
\end{equation}
In the geometric setting, $\theta_j$ and $2\emm$ correspond to the mesh-refinement factor and order of the PDE, respectively. Under standard assumptions for finite element discretizations, classical theory shows that the condition number on a given grid is bounded by a constant (depending on the element order) times $h_{min}^{-2m}$, where $h_{min}$ is the smallest element size on that grid (see \cite[Sec.~5.2]{Strang2008}). Our abstract parameter $h_j$ is therefore bounded by that constant times the grid $j$ mesh size. 

To allow a {\em progressive-precision} V-cycle, where precision is tailored to each grid in the hierarchy, assume now that $\ewed$ varies by letting $\ewed_j$ denote the unit roundoff used on level $j, \, 1 \le j \le \ell$. (We use $\ewed$  without subscripts when the level is understood.) Specifically, $\ewed_j$-precision is used on level $j$ to store the data, perform relaxation,  transfer residuals to level $j-1$ and corrections to level $j + 1$, and round residuals transferred from level $j + 1$. Define the {\em precision coarsening factor} by
\begin{equation}
{\dot\zeta}_j = \frac{\ewed_{j - 1}}{\ewed_j}, \, 2 \le j \le \ell.
\label{zeta}
\end{equation}
To accommodate our use of a geometric series involving the rounding-error effects on each level of the V-cycle, let $\cratio  = \min_{1 \le j \le \ell} \{\theta_j {\dot\zeta}_j^{-\frac{1}{m}}\}$. The estimate in Theorem~\ref{theorem3} suggests that $\cratio$ should be substantially larger than $1$ (i.e, ${\dot\zeta}_j \gg \theta_j^m$) so that the V-cycle convergence factors are bounded nicely in terms of the two-grid factors. However, our only formal assumption for our abstract theory is that $\cratio  > 1$.

Only the low precision varies by level in the V-cycle because its finest level is fixed. But FMG's outer loop uses progressively finer grids for the inner loop's finest levels, thus enabling variable $\ewe_j$ and $\eweb_j$, $1 \le j \le \ell$, where $\ell$ is assumed to be the very finest level used in the FMG scheme.

\section{Convergence Theory: Two-grid}
\label{sec:analysis-twogrid}

In what follows, we develop certain bounds that are written in terms of various parameters such as the convergence estimate in (\ref{gamma}) involving $\rtgstar$. Nothing other than what is initially stated about these bounds is assumed until we make conclusions about when and how well {\ir } actually converges. However, useful prototypes for this work are the matrix equations associated with a hierarchy of grids that arise from discretizing a PDE such as the model biharmonic equation treated in \cite{Benzaken2020}. In this case, we have in mind parameters that are fixed constants so that the bounds hold uniformly in the mesh size. Our references to ``optimal'' and ``optimally'' here are only meant to suggest the loose concept that the corresponding bounds and parameters should be nice in some way. For example, $\rtgstar$  should not be very close to $1$. However, in the PDE context, we mean to suggest that these qualifiers also connote a sense that the bounds and parameters hold uniformly in the mesh size. In any event, our aim is to provide a framework that can be used to confirm optimal performance for specific applications.

Several additional parameters appear below in our convergence estimates, such as the condition number of $P^tP$ denoted by $\kappa(P^tP) = \|P^tP\| \cdot \|(P^tP)^{-1}\|$, and
\[
\alp = \frac{m_P}{1 - m_P\ewed} \textrm{\quad and\quad  }\aladot = \frac{m_A + 1}{1 - (m_A + 1)\ewed}.
\]
To account for rounding errors in relaxation, suppose that a constant $\M$ exists such that computing $\m z$ for a vector $z \in \Re^n$ in $\ewed$-precision yields the result
\begin{equation}
\m  z + \delta_M, \qquad \|\delta_M\| \le \M \ewed \|z\|.
\label{mest}
\end{equation}

Our initial two-grid result assumes that finite precision is used only in the transfers between levels. We ignore the initial rounding step and assume that the coarse-grid solve and all fine-grid computations are done exactly. For this lemma and our second theorem, we only analyze one cycle of {\tg }, with initial guess $y = 0$.  The exact initial algebraic error is therefore just $y - A^{-1}r = - A^{-1}r$, so relaxation yields the new iterate $-\m r$ with error $- GA^{-1}r$, and the {\tg } residual to be transferred to the coarse grid is just $-AGA^{-1}r$. To accommodate progressive precision, assume that the coarse level uses $\ewed_c = {\dot\zeta} \ewed$ precision, where ${\dot\zeta} \ge 1$ (see (\ref{zeta})). 

 \begin{lem}{\em Limited \tg.}
Consider a limited version of the two-grid correction scheme, where all computations are in infinite precision except for $\ewed_c$-precision computation of $b_c$ in step 4 and  $d$ in step 6 of Algorithm~\ref{alg-tg}. Then the result of one such limited {\tg } cycle yields a result $y$ with error $y - A^{-1}r$ that satisfies
\[
\|y - A^{-1}r\|_A \le \rltg \|A^{-1}r\|_A,
\]
where $\rltg = \rtgstar + \pltg, \, \pltg = \rtgstar \ead + 3\ead  + 2\ead^2, \textrm{ and }
\ead = 3 {\dot\zeta} \kappa^\frac{1}{2}(P^tP) \alp \pvard.$
\label{lemma}
 \end{lem}
 \begin{proof}
Superscripts $(0), (\frac{1}{2}),$ and $(1)$ are used here to keep track of the errors $e$ and iterates $y$, with $e^{(0)} = - A^{-1}r$ and $y^{(0)} = 0$. The proof proceeds by treating in turn three cases based on where rounding error is assumed to occur:
\begin{itemize}
\item{\em Case 1 (Restriction)}: All computations are in infinite precision except for finite-precision computation of $b_c$ in step 4 of \tg.
\item{\em Case 2 (Interpolation)}: All computations are in infinite precision except for finite-precision computation of $d$ in step 6 of \tg.
\item{\em Case 3 (Restriction and Interpolation)}: All computations are in infinite precision except for finite-precision computations in steps 4 and 6 of \tg.
\end{itemize}

Beginning with {\em Case 1}, let each entry of the vector $\delta$ be the error in computing the corresponding entry of $P^tr_{\textrm{tg}}$. This task is just an inner product between $P$'s $i^{\textrm{th}}$ column and $r_{\textrm{tg}}$, which, appealing to just the nonzero entries of $P$, are vectors of length at most $m_P$. In analogy to (\ref{fl6}) and remembering that $P$'s entries are nonnegative, $b = 0$, and $r_{\textrm{tg}} = AGe^{(0)}$ is exact, we have the loose bound
\[
|\delta|  \le \left(\ewed_c + (1 + \ewed_c) \alp \ewed \right) P^t |AGe^{(0)}| \le 3 {\dot\zeta} \alp \ewed P^t |AGe^{(0)}|.
\]
Since $ \|P^t\| =  \|P\| = \|P^tP\|^\frac{1}{2}$, we therefore have that
\begin{equation}
\|\delta\| \le 3 {\dot\zeta} \alp \ewed \|P^tP\|^\frac{1}{2} \|AGe^{(0)}\| \le 3 {\dot\zeta} \alp \ewed \|P^tP\|^\frac{1}{2} \|A\|^\frac{1}{2} \|Ge^{(0)}\|_A.
\label{trans}
\end{equation}
Since $\|Ge^{(0)}\|_A \le \|e^{(0)}\|_A$ by assumption, then
\begin{equation}
\|\delta\| \le 3 {\dot\zeta} \alp \ewed \|P^tP\|^\frac{1}{2} \|A\|^\frac{1}{2} \|e^{(0)}\|_A.
\label{delta}
\end{equation}
Remembering that $y^{(\frac{1}{2})} = \m r^{(0)}$ is the intermediate iterate formed by relaxation on $y^{(0)} = 0$, note that the final computed update in the {\tg } cycle is given by
\begin{align*}
y^{(1)} &= y^{(\frac{1}{2})} - d \\
&= y^{(\frac{1}{2})} - P B_c (P^tAP)^{-1}(P^t AGe^{(0)} + \delta) \\
&= \underbrace{y^{(\frac{1}{2})} - P B_c (P^tAP)^{-1}P^t AGe^{(0)}}_\textrm{exact update} -  \underbrace{P B_c (P^tAP)^{-1}\delta}_\textrm{propagated error}.
\end{align*}
Subtracting the exact solution, $A^{-1}r^{(0)}$, of the {\ir } residual equation from both sides and noting that $e^{(\frac{1}{2})} = Ge^{(0)}$ yields
\begin{align*}
e^{(1)} &= e^{(\frac{1}{2})} - P B_c (P^tAP)^{-1}P^t AGe^{(0)} - P B_c (P^tAP)^{-1}\delta\\
&= (I - P B_c (P^tAP)^{-1}P^t A)Ge^{(0)} - P B_c (P^tAP)^{-1}\delta .
\end{align*}
Taking energy norms of both sides yields
\begin{align}
\|e^{(1)}\|_A &\le \|(I - P B_c (P^tAP)^{-1}P^t A)Ge^{(0)}\|_A + \| B_c (P^tAP)^{-1}\delta\|_{A_c}\nonumber \\
&\le \|(I - P B_c (P^tAP)^{-1}P^t A)Ge^{(0)}\|_A + \| B_c \|_{A_c} \| (P^tAP)^{-1}\delta\|_{A_c}\nonumber \\
&\le \|(I - P B_c (P^tAP)^{-1}P^t A)Ge^{(0)}\|_A + 2 \|(P^tAP)^{-\frac{1}{2}}\|\cdot \|\delta\|.
\label{e1temp}
\end{align}
But $\|(P^tAP)^{-\frac{1}{2}}\| \le  \|A^{-\frac{1}{2}}\| \cdot  \|(P^tP)^{-\frac{1}{2}}\| $ follows from noting that if $z \ne 0$ is a coarse-grid eigenvector of $P^tAP$ belonging to the smallest eigenvalue, $\lambda_c$, of $P^tAP$, then the smallest eigenvalue, $\lambda$, of $A$ satisfies
\[
\lambda
\le \frac{\langle APz,Pz\rangle}{\langle Pz,Pz\rangle}
= \frac{\langle APz,Pz\rangle}{\langle z,z\rangle} \cdot \frac{\langle z,z\rangle}{\langle Pz,Pz \rangle }
= \lambda_c \frac{\langle (P^tP)^{-1}Pz,Pz\rangle}{\langle Pz,Pz \rangle }
\le \lambda_c \|(P^tP)^{-1}\|.
\]
Thus, (\ref{gamma}), (\ref{delta}), and (\ref{e1temp}) combine to yield
\begin{equation}
\|e^{(1)}\|_A \le \left(\rtgstar + 2\kappa^\frac{1}{2} (P^tP) \kappa^\frac{1}{2} (A) 3 {\dot\zeta} \alp \ewed \right) \|e^{(0)}\|_A = (\rtgstar + 2\ead)\|e^{(0)}\|_A.
\label{e0}
\end{equation}
The energy convergence factor for {\em Case 1} is therefore bounded by $\ra = \rtgstar + 2\ead$. 

Consider now {\em Case 2}. The symbols we use here are the same as before, but defined differently now to suit this case. Accordingly, let $\delta$ be the error in the correction term $d$ incurred due to $\ewed_c$-precision interpolation of $d_c$ in step 6 of \tg. From (\ref{fl5a}), we then have that $|\delta|  \le  \alp \ewed_c P |d_c|,$ which yields a bound similar to but simpler than (\ref{delta}) because there is no need for rounding: $\|\delta\| \le  \alp \ewed_c \|P^tP\|^\frac{1}{2} \|d_c\|.$ The squared relative energy norm of the error caused by this computed correction is then bounded loosely as follows:
\begin{equation}
\frac{\langle A \delta, \delta\rangle }{\langle APd_c,Pd_c\rangle }
\le \|A\|\cdot \|(P^tAP)^{-1}\|\frac{\|\delta\|^2 }{\|d_c\|^2} \le \kappa(P^tP) \kappa(A) ({\dot\zeta}\alp)^2 \ewed^2 \le \ead^2.
\label{perr}
\end{equation}

With $e^{(0)}$ and $e^{(1)}$ again denoting the respective initial and final errors and $e^{(\frac{1}{2})}$ the error after relaxation, then the computed coarse-grid update becomes
\[
e^{(1)} = \underbrace{e^{(\frac{1}{2})} - P d_c}_\textrm{exact update} + \underbrace{\delta}_\textrm{$\ewed_c$ error}.
\]
Here we invoke the only property of $d_c$ that we need, that is, that its interpolation and correction to the fine-grid iterate reduces the error by a factor of at least $\rtgstar$:
\[
\|e^{(\frac{1}{2})} - Pd_c\|_A \le \rtgstar \|e^{(0)}\|_A.
\]
We thus obtain the following bound on the energy convergence factor:
\begin{align*}
\|e^{(1)}\|_A &= \|e^{(\frac{1}{2})} - P d_c - \delta\|_A\\
&\le  \|e^{(\frac{1}{2})} - P d_c\|_A + \|\delta\|_A\\
&\le \rtgstar \|e^{(0)}\|_A + \ead \|Pd_c\|_A\\
&= \rtgstar \|e^{(0)}\|_A + \ead \|e^{(\frac{1}{2})} - Pd_c - e^{(\frac{1}{2})}\|_A\\
&\le \rtgstar \|e^{(0)}\|_A + \ead\left(\rtgstar \|e^{(0)}\|_A + \|e^{(\frac{1}{2})}\|_A\right)\\
&\le \left(\left(1 + \ead \right) \rtgstar  + \ead\right) \|e^{(0)}\|_A.
\end{align*}

To summarize, the result for {\em Case 1} is that the two-grid solver converges in energy with factor bounded by $\ra = \rtgstar + 2\ead$, while for {\em Case 2} the factor is bounded by $\rb = \left(1 + \ead \right) \rtgstar  + \ead$. But {\em Case 2} made no use of the specific form of $d_c$: the only property of $z$ that was actually used was that it led to a reduction in the energy norm by a factor bounded by $\rtgstar$. So we can easily analyze {\em Case 3} by just replacing $\rtgstar$ by $\ra$ in the expression for $\rb$, thus proving the lemma.

\end{proof}

To rein in complexity in what follows, we are occasionally loose with upper bounds. Accordingly, for the two-grid estimate, let $\cc$ be any constant such that
\begin{equation}
\cc \ge (1 + \ewed) \max \{ \M \|A\|, \A \M, \A \|M\|\}\}.
\label{simpl}
\end{equation}

\begin{thm}{\em \tg.}
One cycle of Algorithm~\ref{alg-tg} applied to (\ref{ayr}) converges according to
\begin{equation}
\|y - A^{-1}r\|_A \le \rtg \|A^{-1}r\|_A, \quad \rtg = \rtgstar + \ptg,
\label{conv}
\end{equation}
provided $\pvard$ is small enough that $\ptg < 1 - \rtgstar$, where $\ptg = \ptg(\pvard)$ written as its linear part plus higher-order  (quadratic and cubic) terms $\hot = \hot(\pvard)$ is given by
\begin{align}
\ptg &= 4\pvard + (2 + \bet) \ead + \hot, \quad
\hot = 2\pvard^2 + (4 + \bet)\ead \pvard + 2 \ead \pvard^2, \nonumber \\
\ead &= 3 {\dot\zeta} \kappa^\frac{1}{2}(P^tP)  \alp \pvard,  \quad \bet = 2 + 3 \cc + 2\aladot (1 + \cc).
\label{ea}
\end{align}
\label{theorem1}
\end{thm}

\begin{proof}
The {\tg } cycle in Algorithm~\ref{alg-tg} uses $\ewed$-precision for all finite-precision computations. Its objective is to approximate the solution $A^{-1}r$ of (\ref{ayr}). Using superscripts as in the previous proof, we first estimate how the {\tg } cycle approximates the solution $A^{-1}r^{(0)}$ of $A y = r^{(0)}$, where $r^{(0)} = r + \delta_0$, $\|\delta_0\| \le  \ewed \|r\|$, is $r$ rounded to $\ewed$-precision. Specifically, this proof is mostly devoted to showing that the final computed approximation $y^{(1)}$ satisfies
\begin{equation}
\|y^{(1)} - A^{-1}r^{(0)}\|_A \le \rtghat \|A^{-1}r^{(0)}\|_A, \quad \rtghat  = \rtgstar + \ptghat,
\label{conv1}
\end{equation}
where 
\begin{equation}
\ptghat = (\pvard + \ead + \ead \pvard) (\rtgstar + 2\bet_0\ead) + 2\bet_0\ead + (1 + \pvard) \ead + (1 + \ewed)\M \|A\|\ewed + \pvard,
\label{rho1}
\end{equation}
with $\bet_0 = 1 + \M \|A\| + \aladot\left(1 + \A \EM + \A \M \ewed \right)$.
We would be done if we could establish (\ref{conv1})-(\ref{rho1}) because we could then use (\ref{gamma}) and the observations that
\[
2\bet_0\ead + (1 + \ewed)\M \|A\|\ewed \le 2\bet_0\ead + (1 + \ewed)\M \|A\|\pvard \le \bet \ead
\]
and $\rtgstar + \bet\ead \le \rtg < 1$ to prove (\ref{conv}) as follows:
\begin{align*}
\|y^{(1)} - A^{-1}r\|_A &\le \|y^{(1)} - A^{-1}r^{(0)}\|_A + \|A^{-1}(r - r^{(0)})\|_A \\
&\le \rtghat \|A^{-1}r^{(0)}\|_A + \|A^{-1}(r - r^{(0)})\|_A \\
&\le \rtghat \|A^{-1}r\|_A + (1 + \rtghat)\|A^{-1}(r - r^{(0)})\|_A \\
&\le \rtghat \|A^{-1}r\|_A + (1 + \rtghat) \|A\|^\frac{1}{2}  \ewed \|r\| \\
&\le \rtghat \|A^{-1}r\|_A +  (1 + \rtghat) \pvard \|A^{-1}r\|_A \\
&= \left(\rtgstar + (1 + \rtgstar) \pvard + (1 + \pvard)\left[(\pvard + \ead + \ead \pvard) (\rtgstar + 2\bet_0\ead) \right.\right. \\
 & \quad \left.\left. + 2\bet_0\ead + (1 + \pvard) \ead + (1 + \ewed)\M \|A\|\ewed + \pvard\right]\right) \|A^{-1}r\|_A \\
&\le \left(\rtgstar + 2 \pvard + (1 + \pvard)\left[\pvard + \ead + \ead \pvard + \bet\ead + (1 + \pvard) \ead + \pvard\right]\right) \|A^{-1}r\|_A \\
&= \rtg \|A^{-1}r\|_A.
\end{align*}

To establish (\ref{conv1})-(\ref{rho1}), we begin by accounting for the rounding-error effects in relaxation in step 2 and residual computation in step 3. Since $y^{(0)} = 0$, (\ref{ayr}) becomes $Ay = - Ae^{(0)}$, where the error in $y^{(0)}$ is $e^{(0)} = y^{(0)} - A^{-1}r^{(0)} = - A^{-1}r^{(0)}$. Using $y^{(0)} = 0$ again shows that relaxation computed in $\ewed$-precision yields
\begin{equation}
y^{(\frac{1}{2})} = \underbrace{\m r^{(0)}}_\textrm{exact iterate} + \underbrace{\delta_M}_\textrm{$\ewed$ error}, \quad \|\delta_M\| \le \M \ewed \|r^{(0)}\|.
\label{y12}
\end{equation}
The next step is to compute the {\tg } residual in $\ewed$-precision for transfer to the coarse grid. Using (\ref{fl5b}), we have that
\begin{equation}
r^{(\frac{1}{2})} =  \underbrace{Ay^{(\frac{1}{2})} - r^{(0)}}_\textrm{exact residual} +  \underbrace{\delta_1}_\textrm{$\ewed$ error}, \quad
|\delta_1| \le \aladot \ewed \left(|r^{(0)}| + |A| \cdot |y^{(\frac{1}{2})}|\right).
\label{resy12}
\end{equation}
But the algebraic errors $y^{(\frac{1}{2})} - A^{-1} r^{(0)}$ and $e^{(0)} = y^{(0)} - A^{-1} r^{(0)}$ are related according to $y^{(\frac{1}{2})} - A^{-1} r^{(0)} = Ge^{(0)} + \delta_M$, so $A y^{(\frac{1}{2})} - r^{(0)} = A(Ge^{(0)} + \delta_M)$. Letting $\delta_2 = A \delta_M + \delta_1$, we can therefore rewrite the computed {\tg } residual as
\[
r^{(\frac{1}{2})} = \underbrace{AGe^{(0)}}_\textrm{exact residual} + \underbrace{\delta_2}_\textrm{propagated $\ewed$ error},
\]
where
\[
|\delta_2| \le |A\delta_M| + |\delta_1| \le |A\delta_M| + \aladot \ewed \left(|r^{(0)}| + |A| \cdot |\m r^{(0)} + \delta_M|\right).
\]
By (\ref{y12}), the observation that $\|r^{(0)}\| = \|Ae^{(0)}\| \le  \|A\|^\frac{1}{2} \|e^{(0)}\|_A$, other similar matrix norm bounds, and the triangle inequality, we then have that
\begin{align*}
\|\delta_2\| &\le \M \ewed \|A\| \cdot  \|r^{(0)}\| +
\aladot \ewed \left(1 + \A \EM + \A \M \ewed \right)\|r^{(0)}\| \\
&\le \left(\M \ewed \|A\| + \aladot \ewed \left(1 + \A \EM + \A \M \ewed \right)\right) \|A\|^\frac{1}{2}\|e^{(0)}\|_A.
\end{align*}
For the rounding error $\delta_3$ incurred in transferring the computed residual to the coarse grid, we can thus mimic (\ref{trans})-(\ref{delta}) with $\delta$ replaced by $\delta_3$ and $AGe^{(0)}$ replaced by $AGe^{(0)} + \delta_2$ to obtain
\[
\|\delta_3\| \le \alp \ewed \|P^tP\|^\frac{1}{2} \|AGe^{(0)} + \delta_2\|,
\]
which leads directly to the loose bound
\begin{equation}
\|\delta_3\| \le \bet_0 \alp \ewed \|P^tP\|^\frac{1}{2} \|A\|^\frac{1}{2} \|e^{(0)}\|_A 
\le \bet_0 3 {\dot\zeta} \alp \ewed \|P^tP\|^\frac{1}{2} \|A\|^\frac{1}{2} \|e^{(0)}\|_A.
\label{gamma3}
\end{equation}
Note that (\ref{gamma3}) is just (\ref{delta}) with the extra leading factor $\bet_0$. We can therefore proceed as we did after (\ref{delta}), but with $\ead$ replaced by $\bet_0\ead$ and the convergence factor corresponding to Case 1 of Lemma \ref{lemma} now reading $\ra =  \rtgstar + 2\bet_0\ead$. The rest of the lemma applies directly with $\ead$ unchanged, so letting $d$ denote the computed correction from the coarse grid, we could then conclude that the exact update $\m r^{(0)} - d$ of the exact intermediate iterate $\m r^{(0)}$ converges to $A^{-1}r^{(0)}$ in energy with factor bounded by $\rc = \left(1 + \ead \right) (\rtgstar + 2\bet_0\ead) + \ead $:
\begin{equation}
\|\m r^{(0)} - d - A^{-1}r^{(0)}\|_A \le \rc \|e^{(0)}\|_A.
\label{rho3}
\end{equation}
However, to bound the error in the computed quantities, we must now account for rounding errors in $y^{(\frac{1}{2})}$ due to the use of $\m $ and in the subtraction $y^{(\frac{1}{2})} - d$.

To this end, note that the error in the computed iterate $y^{(1)}$ can be written as
\[
e^{(1)} = y^{(1)} - A^{-1}r^{(0)} = (\m r^{(0)} + \delta_M - d)(1 + \delta_4) - A^{-1}r^{(0)},
\]
where the rounding error $\delta_4$ due to the subtraction satisfies $|\delta_4| \le \ewed $ (see (\ref{fl1})). Taking energy norms of both sides and rearranging terms yields
\begin{align}
\|e^{(1)}\|_A &= \|(\m r^{(0)} - d - A^{-1}r^{(0)})(1 + \delta_4) + \delta_M(1 + \delta_4) + \delta_4 A^{-1}r^{(0)}\|_A \nonumber\\
&\le \|(\m r^{(0)} - d - A^{-1}r^{(0)})(1 + \delta_4)\|_A + \|\delta_M(1 + \delta_4)\|_A + \|\delta_4 A^{-1}r^{(0)}\|_A .
\label{bnd1}
\end{align}
Using (\ref{rho3}), the first term on the right is bounded according to
\begin{align}
\|(\m r^{(0)} &- d - A^{-1}r^{(0)})(1 + \delta_4)\|_A \nonumber\\
&\le \|(\m r^{(0)} - d - A^{-1}r^{(0)})\|_A + \|(\m r^{(0)} - d - A^{-1}r^{(0)})\delta_4\|_A \nonumber\\
&\le \rc\|e^{(0)}\|_A + \|A^\frac{1}{2}\|\cdot\|(\m r^{(0)} - d - A^{-1}r^{(0)})\delta_4\| \nonumber\\
&\le \rc\|e^{(0)}\|_A + \ewed \|A^\frac{1}{2}\|\cdot\|\m r^{(0)} - d - A^{-1}r^{(0)}\| \nonumber\\
&\le \rc\|e^{(0)}\|_A + \pvard\|\m r^{(0)} - d - A^{-1}r^{(0)}\|_A \nonumber\\
&\le \left(1 + \pvard\right)\rc\|e^{(0)}\|_A.
\label{bnd2}
\end{align}
Since $r^{(0)} = -Ae^{(0)}$, then the remaining two terms are bounded according to
\begin{align}
\|\delta_M(1 + \delta_4)\|_A + \|\delta_4 A^{-1}r^{(0)}\|_A
&\le \|A^\frac{1}{2}\|\left(\|\delta_M(1 + \delta_4)\| + \|\delta_4 A^{-1}r^{(0)}\|\right) \nonumber\\
&\le (1 + \ewed) \ewed \M \|A^\frac{1}{2}\| \cdot \|r^{(0)}\| + \pvard\|A^{-\frac{1}{2}}r^{(0)}\| \nonumber\\
&\le \left((1 + \ewed) \ewed \M \|A\| + \pvard\right)\|e^{(0)}\|_A .
\label{bnd3}
\end{align}
Bounds (\ref{bnd1})-(\ref{bnd3}) combine to establish (\ref{conv1})-(\ref{rho1}) and thus prove the theorem.

\end{proof}

\begin{rem}{\em Mixed Precision Rationale.}
Theorems \ref{theorem2} and \ref{theorem1} suggest why mixed precision is needed and why it works. Both the inner and outer iterations compute residuals: {\tg } in (\ref{resy12}) and {\ir } in (\ref{r0}). However, while the estimates for {\tg } only involve $\kappa^\frac{1}{2}(A)$ so that low precision suffices, the appearance of $\kappa(A)$ in the estimates for {\ir } indicates its need for high precision. The key point here is that because {\tg } starts with a zero initial guess, the {\tg } iterate is just $M$ times the RHS of its target equation (\ref{ayr}). Successive {\ir } iterates can of course be written in terms of $b$ in (\ref{axb}), but that expression becomes increasingly complex as the iterations proceed. Accumulation of the many calculations over successive {\ir } iterations accounts for the need for high precision and explains the success of the mixed-precision solver.
\end{rem}

\section{Convergence Theory : Multigrid V-cycle}
\label{sec:analysis-mg}

Our next theorem extends Theorem~\ref{theorem1} to the multilevel case with progressive precision by confirming that one V-cycle reduces the error optimally toward the solution of the target equation $Ay = r$ provided the derived expression for $\pv$ is less than $1 - \rvstar$. This proviso means that coarsening in the grid hierarchy should be fast enough and progression of the precision should be slow enough that the factors in (\ref{Theta}) and (\ref{zeta}), respectively, should be small enough to ensure that $\cratio  \gg 1$. It further requires that $\kappa(A_j) \ll \ewed_j^{-2}, \, 1 \le j \le \ell$. Theorem~\ref{theorem2} then confirms that the mixed-precision version of {\ir } with a V-cycle as the inner loop converges optimally to the solution of (\ref{axb}) until to the order of the limiting accuracy $\thresh$ is reached. 

To proceed, we need to make assumptions about the parameters in (\ref{ea}) and how they behave on all grid levels. Writing $\alpha_{M_j}$ for the parameter $\M$ on level $j \in \{1, 2, \dots, \ell\}$, it is easy to see that $\alpha_{M_j}$ is of order $\frac{1}{\|A_j\|}$ for Richardson iteration, where $M_j = s_j I_j$: we assume that the energy norm of its error propagation matrix $G_j = I_j - s_j A_j$ is less than $1$, so $s_j$ must be positive and bounded according to
\[
s_j = \frac{\|s_j A_j\|_{A_j}}{\|A_j\|} \le \frac{\|I_j\|_{A_j} + \|I_j - s_j A_j\|_{A_j}}{\|A_j\|} <\frac{2}{\|A_j\|},
\]
from which follows $\alpha_{M_j} < \frac{2}{\|A_j\|}$. For Jacobi, where $G_j = I_j - s_j D_j^{-1}A_j$, $D_j$ is the diagonal of $A_j$, and $s_j = \frac{\omega}{\|D_j^{-1/2}A_jD_j^{-1/2}\|}$, it is straightforward to show that $\alpha_{M_j} = \bigO(\frac{\kappa(D_j)}{\|A_j\|})$. These estimates are typical of relaxation methods applied to elliptic PDEs because they can only effectively target the upper end of the spectrum of $A$ with the simple way they approximate $A^{-1}$. Considering any stronger relaxation scheme that captures the lower end of the spectrum (think of $A^{-1}$ in the extreme case) would beg the question as to how one would solve the resulting linear system that must accompany this stronger method. To exclude such a consideration, assume now that the constant $\cc$ defined in $(\ref{simpl})$ holds on all levels (e.g., $\cc \ge (1 + \ewed_j) \max \{ \alpha_{M_j} \|A_j\|, \alpha_{M_j}\||A_j|\|, \||A_j|\| \|M_j\|\}\}, \, 1 \le j \le \ell$). We also now redefine the following parameters to mean their maxima over all levels:
\begin{align*}
\kappa (P^tP) &= \max_{1 \le j \le \ell} \kappa (P_j^tP_j), \quad \aladot = \max_{1 \le j \le \ell}\frac{m_{A_j}}{1 - m_{A_j} \ewed_j},\quad \alp = \max_{1 \le j \le \ell}\frac{m_{P_j}}{1 - m_{P_j} \ewed_j},
\\
& \alabar = \max_{1 \le j \le \ell}\frac{m_{A_j}}{1 - m_{A_j} \eweb_j}, \quad \textrm{and}\quad \alpnodot = \max_{1 \le j \le \ell}\frac{m_{P_j}}{1 - m_{P_j} \ewe_j}.
\end{align*}

These assumptions are valid for a very general class of elliptic PDEs. They allow us to take the coefficients in the expressions for $\ead$ and $\bet$ in (\ref{ea}) to mean these maxima. Note then that all of the coefficients of $\pvard$ in $\ptg = \ptg(\pvard)$ are positive. The same expressions hold for the $\ead$ and $\bet$ associated with each grid level $j$, but with $\pvard$ replaced by $\pvarj = \kappa^\frac{1}{2}(A_j) \ewed_j$. Let $\ptgjb$ denote $\ptg$ in (\ref{ea}) with the implied argument $\pvard$ replaced by $\pvarj, \, 1\le j \le \ell$. Since $\ptgjb$ consists of a strictly linear term in $\pvarj$ plus higher-order terms, it is easy to see by (\ref{Theta}) and (\ref{zeta}) that 
\begin{equation}
\ptgjmb = \ptg (h^{-\emm}_{j - 1}\ewed_{j - 1}) = \ptg (\theta_j^{-\emm} {\dot\zeta}_j h^{-\emm}_j \ewed_j ) \le \cratio^{-\emm} \ptgjb, \quad 2 \le j \le \ell .
\label{deltaemm}
\end{equation}

\begin{thm}{\em \V.}
If $\pvard_j$ is small enough to ensure that $\pv(\pvard_j) =  \frac{\cratio^\emm}{\cratio^\emm - 1}\ptg(\pvard_j) < 1 - \rvstar, \, 1 \le j \le \ell$, then one cycle of the progressive-precision version of Algorithm~\ref{alg-mg} for solving the {\ir } residual equation converges according to
\begin{equation}
\|y - A^{-1}r\|_A \le \rv \|A^{-1}r\|_A, \, \rv = \rvstar + \pv, \,
\pv = \pv(\pvard_\ell) = \frac{\cratio^\emm}{\cratio^\emm - 1}\ptg(\pvard_\ell) .
\label{rh-v}
\end{equation}
\label{theorem3}
\end{thm}

\begin{proof}

Note from (\ref{vb}) and the definition of $T_j$ that
\[
V_j = (I_j - P_j (I_{j - 1} -V_{j - 1}) A_{j - 1}^{-1}P_j^t)G_j,
\]
which is the error propagation matrix for {\tg } on the ``finest'' grid $j$ with $B_{j - 1} = I_{j - 1} - V_{j - 1}$. Thus, step 6 of {\V } is step 5 of {\tg } with this choice of $B_c = I_{j - 1} - V_{j - 1}$ and {\V } is {\tg } with the coarse level solved recursively. Note that we satisfy our requirement on $B_c$ here because $\|B_{j - 1} - I_{j - 1}\|_{A_{j - 1}} = \|V_{j - 1}\|_{A_{j - 1}}  \le \rvstar < 1$. Note also that $\rvstar$ serves as the {\tg }  convergence bound for $\|V_j\|$, so we can take $\rtgstar = \rvstar$.

We could now use (\ref{convv}) and Theorem~\ref{theorem2} to account for the effects of $\ewed_\ell$-precision on the finest grid, but we also need to assess the effects of the rounding error accumulated from all of the coarse levels. This we do by recursive use of bound (\ref{conv}), exploiting its form as a perturbation from the infinite-precision bound. A key point that enables this recursion is that $\ptg(\pvard_j)$ does not depend on $\rtgstar$ if $\rtg(\pvard_j) < 1$, which holds because of our assumption on $\pvard_j$ and because $\rtgstar = \rvstar$. Another key point is that we keep these perturbations separate on all levels by writing the error propagation matrices in perturbation form. More precisely, for each $j \in \{1, 2, \dots , \ell \}$, let the error propagation matrix for the result of the computed V-cycle be denoted by $V_j + E_j$, where $V_j$ remains as the exact error propagation operator and $E_j$ incorporates the rounding errors it accumulates. To assess the impact that $E_j$ has on $ \rvstar$ in (\ref{convv}), from (\ref{vb}) and the fact that $P_jA_{j - 1}^{-1}P_j^tA_j$ is an energy-orthogonal projection onto the range of $P_j$, we have that the exact level $j$ V-cycle error propagation matrix based on the computed coarse-grid V-cycle has the new bound
\begin{align}
\|V_j\|_{A_j} &= \|(P_j (V_{j - 1} + E_{j - 1}) A_{j - 1}^{-1}P_j^tA_j + T_j) G_j\|_{A_j} \nonumber\\
&\le  \|(P_j V_{j - 1} A_{j - 1}^{-1}P_j^tA_j + T_j) G_j\|_{A_j} + \|P_j E_{j - 1} A_{j - 1}^{-1}P_j^tA_j\|_{A_j} \|G_j\|_{A_j} \nonumber\\
&\le   \rvstar + \|E_{j - 1} A_{j - 1}^{-1}P_j^tA_j\|_{A_{j - 1}} \nonumber\\
&\le   \rvstar + \|E_{j - 1}\|_{A_{j - 1}} \|P_jA_{j - 1}^{-1}P_j^tA_j\|_{A_j} \nonumber\\
&\le  \rvstar + \|E_{j - 1}\|_{A_{j - 1}}.
\label{recur}
\end{align}

We first need to estimate $\|E_j\|_{A_j}$ for the coarsest level ($j = 1$), which only uses relaxation to approximate the solution of $A_1y_1 = r_1$. To this end, using iteration superscripts as in the previous two proofs, note by the definition of $\M$ that the final error on the coarsest grid computed in $\ewed_1$-precision satisfies
\[
e_1^{(1)} = y_1^{(\frac{1}{2})} - A_1^{-1}r_1^{(0)} = \m_1 r_1^{(0)}  - A_1^{-1}r_1^{(0)} + \delta, \quad |\delta| \le \alpha_{M_1} \ewed_1 |r_1^{(0)}| .
\]
Remembering that $\pvarj = \kappa^\frac{1}{2}(A_j) \ewed_j$, we thus have that
\[
\|e_1^{(1)}\|_{A_1} \le \|\m_1 r_1^{(0)}  - A_1^{-1}r_1^{(0)}\|_{A_1} + \|\delta\|_{A_1} \le ( \rvstar  + \alpha_{M_1} \pvarone)\|e_1^{(0)}\|_{A_1} ,
\]
from which follows the bound
\begin{equation}
\|E_1\|_{A_1} \le \alpha_{M_1} \pvarone .
\label{E1}
\end{equation}

Bound (\ref{recur}) implies that the basic convergence factor that we should be using in (\ref{conv}) is not just $ \rvstar$, but rather $ \rvstar + \|E_{j - 1}\|_{A_{j - 1}}$, so we can now use Theorem~\ref{theorem1} to assert that $\rh_v =  \rvstar + \|E_{j - 1}\|_{A_{j - 1}} + \ptg$ bounds the level $j$ V-cycle convergence factor. Restating this conclusion in terms of the rounding-error matrices yields
\begin{equation}
\|E_j\|_{A_j} \le \|E_{j - 1}\|_{A_{j - 1}} + \ptgjb.
\label{recurs}
\end{equation}
This inequality provides the recursion we need. Noting that (\ref{E1}) and (\ref{ea}) lead to the loose bound $\|E_1\|_{A_1} \le \ptgoneb$, then we can start with $j= \ell$ and recurse (\ref{recurs}) back down through the coarse grids, using (\ref{deltaemm}) to conclude that
\[
\|E_\ell\|_{A_\ell} \le \|E_{\ell - 1}\|_{A_{\ell - 1}} + \ptg \le \sum\limits_{j=1}^\ell \ptgjb \le \ptg \sum\limits_{j=1}^\ell \cratio^{m(\ell - j)}
\le \frac{\cratio^\emm}{\cratio^\emm - 1} \ptg ,
\]
thus proving the theorem.
\end{proof}

\section{Full multigrid algorithm}
\label{sec:fmg}

While stopping criteria for general iterative solvers for linear equations is not entirely settled and is often application dependent, matters clarify for discretized PDEs. If  $A_j x_j = b_j$ represent increasingly accurate approximations to a given PDE, then the error in the numerical solution of (\ref{axb}) incorporates errors from algebraic and discretization processes (not to mention other sources). Investing effort to compute an approximation with an algebraic error that is much smaller than the discretization error is usually not productive. Full multigrid (see \cite{tutorial}) is based on this premise. Properly designed and applied, full multigrid is a direct method in the sense that it targets the PDE: a single iteration through its outer loop guarantees a solution with algebraic accuracy comparable to that of the discretization. It can be the most efficient solver for this purpose, with optimal complexity proportional to the size of the finest grid. The key to its success is to capitalize on a consistent form of discretization accuracy in the multigrid hierarchy and to aim at delivering an approximation that is comparable to that accuracy on every level. The basic full multigrid approach is to start on the coarsest grid and proceed to the finest, making sure that each grid along the way is solved by V-cycles with accuracy commensurate to the discretization. In essence, if grid $j - 1$ is solved to within an error of $Ch_{j - 1}^q$ for some positive constants $C$ and $q $, then using that result as an initial guess on grid $j$ means that the initial error on grid $j$ is bounded by some small multiple (depending on $\theta_j$) of $Ch_j^q$. This in turn means that only a small number of V-cycles are needed to obtain discretization accuracy on grid $j$ (i.e., error below $Ch_j^q$).

The full multigrid algorithm based on $\I \ge 1$ inner {\ir } cycles using one {\V } each is given by the pseudocode described {\fmg} below. Note that {\fmg}  amounts to three nested loops: outer {\fmg}, middle {\ir}, and inner \V. The choice of $\I$ here is critical to guarantee convergence to within discretization accuracy on each level. The goal of the next section is to determine $\I$ in the presence of rounding errors. 

\noindent
\begin{center}
\begin{minipage}{0.97\linewidth}
\begin{algorithm}[H]
\footnotesize
\caption{$\texttt{FMG$(1,0)$-Cycle (\fmg)}$}
\label{alg-fmg}
\begin{algorithmic}[1]
\Require A, b, P, $\I \ge 1$ {\ir } cycles (using one V$(1,0)$ each), $\ell \ge 1$ {\fmg } levels.
\State $x \leftarrow 0$\COMMENT{Initialize {\fmg } }
\If {$\ell > 1$}  \COMMENT{Check for Coarser Grid}
\State \textcolor{med}{$x_{\ell - 1} \leftarrow $\fmg$(A_{\ell - 1}, b_{\ell - 1}, P_{\ell - 1}, \ell - 1, \I)$}\COMMENT{\textcolor{med}{Compute Coarse-Grid Approximation}}
\State \textcolor{med}{$x \leftarrow P x_{\ell - 1}$}\COMMENT{\textcolor{med}{Interpolate Approximation to Fine Grid}}
\EndIf
\State $i \leftarrow 0$\COMMENT{Initialize {\ir } }
\While {$i < \I$}
\State \textcolor{med}{$r \leftarrow$}\textcolor{red}{$ Ax - b$} \COMMENT{\textcolor{red}{Update {\ir } Residual} and \textcolor{med}{Round}}
\State \textcolor{lo}{$y \leftarrow $\V$(A, r, P, \ell)$}\COMMENT{\textcolor{lo}{Compute Correction by \V}}
\State $i \leftarrow i + 1$\COMMENT{Increment {\ir } Cycle Counter}
\State \textcolor{med}{$x \leftarrow x - y$}\COMMENT{\textcolor{med}{Update Approximate Solution of $Ax = b$}}
\EndWhile
\State\Return $x$\COMMENT{Return Approximate Solution of $Ax = b$}
\end{algorithmic}
\end{algorithm}
\end{minipage}
\end{center}
\vspace{1em}

\section{Convergence Theory: Full Multigrid}
\label{sec:analysis-fmg}

While the theory developed here has elliptic PDEs in mind, the intent is to provide a more general algebraic theory that is removed from specific applications. To obtain such an abstract theory for {\fmg } applied to (\ref{axb}), we need to extract a sense of discretization accuracy directly from the matrix hierarchy itself. To do so, we need to establish a relationship between the exact solutions $A_j^{-1}b_j$ on all levels $j \in \{1, 2, \dots, \ell\},$ as we do in (\ref{abc}) below. Establishing how well grid $j - 1$ solutions approximate grid $j$ solutions then leads us to a sense of discretization accuracy on the finest level, as expressed in (\ref{disc}) below.

To this end and commensurate with our assumption on the exactness of the Galerkin condition (\ref{galerk}), assume now that $b_{j - 1} = P_j^tb_j, \, 2 \le j \le \ell$, are also computed exactly. Since $P_jA_{j - 1}^{-1}P_j^t A_j$ is an energy-orthogonal projection, we have that
\begin{align}
\|A_{j - 1}^{-1}b_{j - 1}\|_{A_{j - 1}} &= \langle A_{j - 1}^{-1}P_j^tb_j, P_j^tb_j\rangle^\frac{1}{2} \nonumber \\
&= \langle (P_jA_{j - 1}^{-1}P_j^t A_j) A_j^{-1} b_j, A_j A_j^{-1} b_j\rangle^\frac{1}{2} \nonumber \\
&= \|(P_jA_{j - 1}^{-1}P_j^t A_j) A_j^{-1} b_j\|_{A_j} \nonumber \\
&\le \|A_j^{-1} b_j\|_{A_j} .
\label{abc}
\end{align}
We can thus characterize the relative accuracy of adjacent levels in the grid hierarchy by assuming that $C$ and $q$ are positive constants such that
\begin{equation}
\|P_j A_{j - 1}^{-1}b_{j - 1} - A_j^{-1}b_j\|_{A_j} \le C h_{j - 1}^q \|A_j^{-1}b_j\|_{A_j}, \quad 2 \le j \le \ell .
\label{C}
\end{equation}
For a fully regular PDE of order $2\emm$, sufficient smoothness of the solution $u$, and finite element functions of a fixed order $k$ (e.g., polynomials of fixed degree $p=k - 1$), it can be shown (see \cite[Sec.~2.2]{Strang2008}) that
\begin{equation}
\|u - u_h\|_s = \mathcal{O}\left(h^{k-s}+ h^{2(k-m)}\right) ,
\label{strangfixconvergence}
\end{equation}
where $u_h$ denotes the discrete solution and $s$ is the highest-order derivative in the Sobolev or Sobolev-equivalent norm $\| \cdot \|_s$. For the energy norm, $s=\emm$, so we assume that $k > \emm$ to ensure convergence. The error is therefore dominated by the first term such that the relative accuracy of the discrete approximation on grid $j$ is $h_j^q$, where $q = k - \emm$. It is easy to show by the triangle inequality that this confirms that (\ref{C}) holds in this case. This bound can also be established from existing matrix theory without reference to a PDE. Specifically, remembering the definitions of $h_{j - 1}$ and $\theta_j$, then (\ref{C}) holds with $C = \theta_j^{-\emm} C_1C_2$ and $q = \emm + 1$ directly from the so-called {\em strong approximation property} (cf. \cite{Vassilevski2008}) given by
\begin{equation}
\|P_j A_{j - 1}^{-1}b_{j - 1} - A_j^{-1}b_j\|_{A_j} \le \frac{C_1}{\|A_j^\frac{1}{2}\|} \|A_j(A_j^{-1}b_j)\| ,
\label{sap}
\end{equation}
which is assumed to hold for any RHS, and the {\em algebraic smoothness property} of our specific RHS given by
\begin{equation}
\|b_j\| = \|A_j(A_j^{-1}b_j)\| \le \frac{C_2}{\|A_j^{-\frac{1}{2}}\|} \|A_j^{-1}b_j\|_{A_j}.
\label{algsmooth}
\end{equation}
The strong approximation property says loosely that $A_j^{-1}b_j$ can be accurately approximated when its expansion in terms of the eigenvectors of $A_j$ are predominantly in the lower spectrum of $A_j$, which is what algebraic smoothness asserts. 

While (\ref{C}) characterizes the relative error in a coarse-grid solution with respect to the next finer grid, it also suggests the following definition. We say that $x_j$ solves $A_jx_j =b_j$ {\em to the order of discretization error} or simply {\em to discretization accuracy} if
\begin{equation}
\|x_j - A_j^{-1}b_j\|_{A_j} \le C h_j^q  \|A_j^{-1}b_j\|_{A_j} , \quad 1 \le j \le \ell.
\label{disc}
\end{equation}
Assume that (\ref{disc}) holds on the coarsest level as a result of one {\fmg } cycle there, that is, when $x_1$ is the result of $\I$ cycles of {\ir}-{\V } on level $j=1$ starting with a zero initial guess. (Note that {\V } consists of just one relaxation sweep.) This assumption holds (similar to (\ref{convv}) with $j=1$) in the common case that $A_1$ is well-conditioned enough that relaxation converges quickly and rounding effects are negligible. 
\begin{thm}{\em \fmg.}
Assume that $\rh_v + \pir < 1$ and that $\thresh$ is small enough and $\I$ is large enough that the following  holds on all levels $j \in \{1, 2,\dots, \ell\}$:
\begin{equation}
(\rh_v + \pir)^\I \left((\sqrt{2} + \ea) \theta^q C h^q +  \ea\right) + \frac{\thresh}{1 -(\rh_v + \pir)} \le C h^q ,
\label{mess}
\end{equation}
where $h = h_j$, $\theta = \theta_j$, and $\ea = \ea_{j} =  \kappa^{\frac{1}{2}} (P_j^tP_j) \alpnodot \pvar_j$, and (with subscript $j$ understood) the parameters $\rh_v$, $\pir$, and $\thresh$ are given by (\ref{rh-v}), (\ref{ec}),  and (\ref{chi}), respectively. Then Algorithm~\ref{alg-fmg} solves (\ref{axb}) to the order of discretization error on each level. 
\label{theorem4}
\end{thm}

\begin{proof}
We proceed by induction. Knowing that (\ref{disc}) holds for the coarsest level $j = 1$ by assumption, suppose that $j > 1$ and that (\ref{disc}) holds on level $j - 1$, where $x_{j - 1}$ is now the result of {\fmg } applied on that level. We need to account for the errors in transferring  $x_{j - 1}$ to the fine grid, but we first follow the effects of exact computation. To this end, note that
\begin{align*}
\langle P_j x_{j - 1} - P_j A_{j - 1}^{-1}b_{j - 1}, &A_j(P_j A_{j - 1}^{-1}b_{j - 1} - A_j^{-1}b_j)\rangle \\
&= \langle x_{j - 1} - A_{j - 1}^{-1}b_{j - 1}, P_j^tA_j P_j A_{j - 1}^{-1}b_{j - 1} - P_j^tA_j A_j^{-1}b_j\rangle_{A_j} \\
&= \langle x_{j - 1} - A_{j - 1}^{-1}b_{j - 1}, b_{j - 1} - P_j^tb_j\rangle = 0 .
\end{align*}
This orthogonality, the induction hypothesis, and (\ref{abc}) allow us to conclude that
 \begin{align}
 \|P_j x_{j - 1} - A_j^{-1}b_j\|_{A_j} &= \|P_j x_{j - 1} - P_j A_{j - 1}^{-1}b_{j - 1} + P_j A_{j - 1}^{-1}b_{j - 1} - A_j^{-1}b_j\|_{A_j} \nonumber \\
&= (\|P_j x_{j - 1} - P_j A_{j - 1}^{-1}b_{j - 1}\|_{A_j}^2 + \|P_j A_{j - 1}^{-1}b_{j - 1} - A_j^{-1}b_j\|_{A_j}^2)^\frac{1}{2} \nonumber \\
 &\le C h_{j - 1}^q (\|A_{j - 1}^{-1}b_{j - 1}\|_{A_{j - 1}}^2 + \|A_j^{-1}b_j\|_{A_j}^2)^\frac{1}{2} \nonumber \\
 &\le \sqrt{2} C h_{j - 1}^q \|A_j^{-1}b_j\|_{A_j} .
 \label{init1}
 \end{align}

We can account for rounding error in $P_j x_{j - 1}$ with logic similar to that leading up to (\ref{perr}), but now in $\ewe$-precision: we can write its computed value as $P_j x_{j - 1} + \delta_j,  \, \|\delta_j\|_{A_j} \le \ea_j \|x_{j - 1}\|_{A_{j - 1}}$. The induction hypothesis and (\ref{abc}) again yields
\begin{align*}
\|x_{j - 1}\|_{A_{j - 1}} &\le \|A_{j - 1}^{-1}b_{j - 1}\|_{A_{j - 1}} + \|x_{j - 1} - A_{j - 1}^{-1}b_{j - 1}\|_{A_{j - 1}} \\
&\le (1 + C h_{j - 1}^q) \|A_{j - 1}^{-1}b_{j - 1}\|_{A_{j - 1}} \\
&\le (1 + C h_{j - 1}^q) \|A_j^{-1}b_j\|_{A_j} ,
\end{align*}
which, with (\ref{init1}) and (\ref{Theta}), leads to
 \begin{align*}
 \|P_j x_{j - 1} + \delta_j - A_j^{-1}b_j\|_{A_j} &\le \sqrt{2} C h_{j - 1}^q \|A_j^{-1}b_j\|_{A_j} + \|\delta_j\|_{A_j} \\
 &\le (\sqrt{2} C h_{j - 1}^q + (1 + C h_{j - 1}^q) \ea_{j})\|A_j^{-1}b_j\|_{A_j} \\
 &\le  \left((\sqrt{2} +  \ea_{j}) \theta_j^q C h_j^q + \ea_{j} \right)\|A_j^{-1}b_j\|_{A_j} .
 \end{align*}
Applying $\I$ cycles of {\ir } to the initial approximation $P_j x_{j - 1} + \delta_j $ thus yields
\[
\|x_j - A_j^{-1}b_j\|_{A_j} \le c_j  \|A_j^{-1}b_j\|_{A_j} ,
\]
where appealing to Theorems~\ref{theorem2} and \ref{theorem3} allows us to take
\[
c_j = (\rh_v + \pir)^\I  \left((\sqrt{2} +  \ea_{j}) \theta_j^q C h_j^q + \ea_{j} \right) + \frac{\thresh}{1 -(\rh_v + \pir)} ,
\]
thus proving the theorem.
\end{proof}
\begin{rem}{\em Understanding the Error Bounds.}
While the proof of Theorem~\ref{theorem4} is fairly simple, its statement nevertheless involves the many parameters that are needed to account for errors in all three of the nested solver loops. We can, however, parse condition (\ref{mess}) somewhat by writing it loosely as follows, assuming that $\ewed \ge \ewe \ge \eweb$ and remembering that $\pvard = \kappa^{\frac{1}{2}} (A) \ewed$ and $\pvar = \kappa^{\frac{1}{2}} (A) \ewe$:
\begin{equation}
(\rvstar + \bigO(\pvard))^\I \left(\sqrt{2} \theta^q C h^q + \bigO(\pvar)\right) +  \bigO(\pvar) \le Ch^q  ,
\label{mess1}
\end{equation}
where the constants implicit in the order symbols $\bigO(\cdot)$ are assumed to be of modest size and where we replaced $\thresh$ by $\bigO(\pvar)$. (We require $\thresh$ to be of order $\pvar$ to ensure convergence of {\ir}.) We must have $\bigO(\pvard) < 1 - \rvstar$ to satisfy Theorem~\ref{theorem4}. We also require $\pvar$ to be bounded by a small constant times $h^q$. (Otherwise, (\ref{mess1}) could not be obtained no matter how many {\ir } iterations are taken.)  This assumption means that the term $\sqrt{2} \theta^q C h^q + \bigO(\pvar)$ in (\ref{mess1}) is bounded by a modest multiple of $Ch^q$, which in turn means that $(\rvstar + \bigO(\pvard))^\I$ only need be small enough to reduce that multiple to below $Ch^q$ itself. On the other hand, if $\bigO(\pvard)$ and $\bigO(\pvar)$ are small enough to be negligible in (\ref{mess1}), then this condition is satisfied by $\I$ being only so large that $(\rvstar)^\I  \sqrt{2} \theta^q < 1$, that is, 
$
\I > \frac{0.5 + q\log_2(\theta)}{| \log_2(\rvstar)|}.
$
This requirement is independent of $h$, which ensures that {\fmg } converges optimally in the sense that it obtains accuracy to the level of the discretization in a uniformly bounded number of floating point operations per fine-grid unknown.
\label{fmg-conv}
\end{rem}

\section{Conclusions}
\label{sec:conclusion}

This paper has established an abstract theory for analyzing the effects of rounding error on two-grid, V-cycle, and FMG solvers applied to sparse SPD matrices. We have not accounted for the effects of inexact system matrices, RHS's, and interpolation operators, but these issues are treated in the related paper \cite{Benzaken2020} along with numerical results for a model biharmonic equation. To avoid further complexity, we have also not assessed the effects of overflow or underflow, post smoothing, or non-polynomial smoothers such as Gauss-Seidel. Nevertheless, the abstract algebraic framework we have developed applies to a large class of applications, including many elliptic PDEs and inherently discrete matrix equations. We have shown that, in the energy norm, the computed approximation is affected more by convergence stalling at limiting accuracy rather than by convergence factors degrading. In fact, the precision used in iterative refinement's inner solve has little effect on the resulting accuracy compared to that of the limiting accuracy determined by the higher precision. The results also show that V-cycles and FMG are capable of leveraging progressive precision by using increasingly lower precision on levels that are increasingly coarser, and thus decreasingly accurate. In this way, progressive precision allows FMG to obtain discretization accuracy using a minimal amount of resources.

\bibliographystyle{siamplain}
\bibliography{mpmg}

\end{document}


\maketitle

\section{Double relaxation sweeps}
The theory in the main paper is restricted to V$(1,0)$-cycles, meaning that each {\V } uses one {\em pre-smoothing} sweep on the way down through the coarse grids and no {\em post-smoothing} sweeps on the way back up to the finest. This specific cycle simplifies the analysis because post-smoothing sweeps substantially complicate the estimates by the accumulation of errors in the transfer of residuals to the coarse levels and, to a lesser extent, because a single pre-smoothing sweep is easier to analyze than multiple sweeps. On the other hand, while the analysis of a general V-cycle would be too complicated to present here, we can more easily analyze multiple pre-smoothing sweeps as we illustrate now for a V$(2, 0)$-cycle. 

A relatively simple way to handle multiple sweeps is to combine them into one. To this end, for each $j \in \{1,2\}$, consider a monotonically energy-convergent stationary linear iteration $x \leftarrow x - \m^{(j)} (Ax - b)$, where $\m^{(j)} \in \Re^{n \times n}$, and let $\alpha_j$ be a constant such that computing $\m^{(j)} z$ for a vector $z \in \Re^n$ in $\ewed$-precision yields the result $\m^{(j)}  z + \delta, \, \|\delta\| \le \alpha_j \ewed \|z\|$. Then the key point here is that the error propagation matrix for relaxation with preconditioner $\mm$ followed by relaxation with preconditioner $\n$ can be written as $(I - \n A)(I - \mm A) = I - \Msq A$, where $\Msq = \mm + \n - \n A\mm$. We can therefore think of, and implement, two relaxation sweeps as just the {\em one} sweep $y \leftarrow y - \Msq(Ay - r)$, which means that we can analyze a V$(2, 0)$-cycle as just a V$(1,0)$-cycle with this $\Msq$.  The implication is that we just need to provide estimates for $\|\Msq\|$ and a constant $\alpha_{\Msq}$ such that computing $\Msq z$ for any vector $z$ in $\ewed$-precision yields the result $\Msq  z + \delta_{\Msq}, \, \|\delta_{\Msq}\| \le \alpha_{\Msq} \ewed \|z\|$. This is done in our next theorem.

\begin{thm}{\em Double Sweeps.}
For a double sweep in the ordering specified by $\Msq z = ((\n z) + (\mm z) - (\n (A(\mm z))))$, the constant $\alpha_{\Msq}$ can be chosen as follows:
\begin{align}
\alpha_{\Msq} &= \|\Msq\| + (1 + \ewed) [(\EN + \alpha_2 \ewed)(\|A\| \alpha_1 + \A \dot{m}_A  \alpha_1 \ewed + \A \dot{m}_A  \|\mm\|) \nonumber \\
&+ \|A\| \cdot \|\mm\| \alpha_2 + \|\mm\| + \EN + 2\alpha_1 + 2\alpha_2].
\label{msq}
\end{align}
Moreover, in general, $\|\Msq\|  \le \|\mm\| + \EN + \|A\| \cdot \|\mm\|\cdot\EN$. The special case $\mm = \n = \frac{\omega}{\|A\|} I, \, 0 < \omega < 2$ yields the sharp estimate $\|\Msq\|  \le \frac{2\omega}{\|A\|}$.
\end{thm}
\begin{proof}
The following is meant to clarify the stages for computing $\Msq z$, with the subscripted $w$'s standing for the indicated quantities computed in $\ewed$-precision:
\[
\Msq z = \underbrace{(\underbrace{(\n z)}_{w_4} + \underbrace{(\mm z)}_{w_1} - \underbrace{(\n \underbrace{(A\underbrace{(\mm z) }_{w_1})}_{w_2})}_{w_3}) .}_{w_5}
\]
Thus, by definition,
\begin{equation}
w_1 = \mm z + \delta_1, \quad \|\delta_1\| \le \alpha_1 \ewed \|z\|.
\label{Mz}
\end{equation}
We then use (\ref{fl5a}) in the main paper to obtain the slight overestimate
\[
w_2 = A\mm z + A \delta_1 + \delta_2, \quad \|\delta_2\| \le \A \dot{m}_A  \ewed \|w_1\| = \A \dot{m}_A  \ewed \|\mm z + \delta_1\|.
\]
Letting $\delta_3 = A \delta_1 + \delta_2$, where $\|\delta_3\| \le \|A\| \cdot \|\delta_1\| +\A \dot{m}_A  \ewed (\|\delta_1\| + \|\mm z\|)$, yields
\[
w_2 = A\mm z + \delta_3, \quad \|\delta_3\| \le \left((\|A\| + \A \dot{m}_A  \ewed) \alpha_1 + \A \dot{m}_A  \|\mm\|\right) \ewed \|z\|.
\]
Similarly,
\[
w_3 =  \n A\mm z + \n \delta_3 + \delta_4,  \quad \|\delta_4\| \le \alpha_2 \ewed \|w_2\| = \alpha_2 \ewed \|A\mm z + \delta_3\|.
\]
Letting $\delta_5 = \n \delta_3 + \delta_4$, where $\|\delta_5\| \le \EN \cdot \|\delta_3\| +\alpha_2 \ewed (\|A\mm z\| + \|\delta_3\|)$, yields $w_3 = \n A\mm z + \delta_5, \, \|\delta_5\| \le \Upsilon \ewed \|z\|$, where
\begin{equation}
\Upsilon  = (\EN + \alpha_2 \ewed)(\|A\| \alpha_1  + \A \dot{m}_A  \alpha_1  \ewed + \A \dot{m}_A  \|\mm\|) +  \|A\| \cdot \|\mm\| \alpha_2 \ewed \|z\|.
\label{MAMb}
\end{equation}
We can now use the estimate $w_4 = \n z + \delta_6, \, \|\delta_6\| \le \alpha_2 \ewed \|z\|$ together with (\ref{Mz}) and (\ref{MAMb}) to obtain
\begin{align*}
w_5 &= (w_1 + w_4 - w_3) + \delta_7 + \delta_8 = \Msq z + \delta_1 - \delta_5 + \delta_6 + \delta_7 + \delta_8, \\
& \|\delta_7\| \le \ewed \|w_1 + w_4\| \le (\|\mm\| + \EN + \alpha_1  + \alpha_2) \ewed\|z\|, \\
& \|\delta_8\| \le \ewed \|w_1 + w_4 - w_3 + \delta_7\| \le  \ewed (\|\Msq\| \|z\|+ \|\delta_1\| + \|\delta_5\| + \|\delta_6\| + \|\delta_7\|).
\end{align*}
Letting $\delta = \delta_1 - \delta_5 + \delta_6 + \delta_7 + \delta_8$, then $w_5 = \Msq z + \delta,$ where
\begin{align*}
\| \delta \| & \le  \|\delta_1\| + \|\delta_5\| + \|\delta_6\| + \|\delta_7\| + \|\delta_8\|\\
& \le \left(\ewed \|\Msq\| + (1 + \ewed) ( \|\delta_1\| + \|\delta_5\| + \|\delta_6\| + \|\delta_7\|\right) \|z\|\\
& \le \left(\|\Msq\|  + (1 + \ewed) \left[ \alpha_1  + \alpha_2 + \Upsilon  + \alpha_1  + \alpha_2\right] \right)\ewed \|z\| ,
\end{align*}
thus establishing (\ref{msq}). The estimates for $\|\Msq\|$ are straightforward.
\end{proof}

\section{Second-order Chebyshev iteration}
The equation in (\ref{msq}) can be used in a recursive way to analyze any Krylov method, where the error propagation matrix is a polynomial in $A$. For example, it is fairly straightforward to show that $\alpha_{\Msq}  = \bigO(\frac{\dot{m}_A}{\|A\|})$ for the $\mathcal{K}^\textrm{th}$-order Chebyshev relaxation (cf., \cite{Adams03}), although the constant in this order bound depends exponentially on $\mathcal{K}$. On the other hand, a more direct approach can achieve a somewhat tighter bound, as illustrated for the case $\mathcal{K} = 2$ in our next theorem.

Second-order Chebyshev relaxation can be formed from two sweeps of Richardson iteration with error propagation factors of the form $I - s_j A, \, j=1,2.$ Assume that the coefficients in the Chebyshev factors are chosen to so that $0 < s_j = \bigO(\frac{1}{\|A\|}), \, j=1, 2$. This assumption would generally hold in the multigrid context when the smoothing interval is chosen to be a fixed percentage of the upper spectrum of $A$. We can thus write Chebyshev iteration in the form $y \leftarrow y - \m_C(Ay - r)$, where $\m_C = \omega_1I - \omega_2A$, $\omega_1 = \bigO(\frac{1}{\|A\|}) > 0$, and $\omega_2 = \bigO(\frac{1}{\|A\|^2}) > 0$.  Define $\dot{m}_A  = \frac{m_A}{1 - m_A \ewed}$, where $m_A$ is the maximum number of nonzeros in the rows of $A$. Finally, suppose that computing $\m_C z$ in $\ewed$-precision for any vector $z$ yields $\m_C  z + \delta_{\m_C}, \, \|\delta_{\m_C}\| \le \alpha_{\m_C} \ewed \|z\|$, for some constant $\alpha_{\m_C}$.

\begin{thm}{\em Chebyshev.}
For one second-order Chebyshev iteration in the ordering specified by $\m_C z = (\omega_1 z) - (\omega_2 (A z))$, we can choose
\[
\alpha_{\m_C} = \|\m_C\| + (\omega_1 + (1 + \ewed)\omega_2 \A \dot{m}_A  + \omega_2\|A\|)(1 + \ewed) .
\]
Note that $\|\m_C\|  = \bigO(\frac{1}{\|A\|})$ and, if $\A \approx \bigO(\|A\|)$, then $\alpha_{\m_C} = \bigO\left(\frac{\dot{m}_A}{\|A\|}\right)$.
\label{thm:Chebyshev}
\end{thm}

\begin{proof}
Computing $\m_C z$ according to $\m_C z = \underbrace{(\underbrace{(\omega_1 z)}_{w_1} - \underbrace{(\omega_2 \underbrace{(A z)}_{w_2})}_{w_3})}_{w_4}$, we have
\begin{align*}
w_1 &= \omega_1 z + \delta_1, \quad \|\delta_1\| \le \omega_1 \ewed \|z\|, \\
w_2 &= A z + \delta_2, \quad \|\delta_2\| \le \A \dot{m}_A  \ewed \|z\|, \\
w_3 &= -\omega_2 w_2 + \delta_3, \quad \|\delta_3\| \le \omega_2 \ewed \|A z + \delta_2\| \le  \omega_2 \ewed (\|A\| + \A \dot{m}_A  \ewed)\|z\|, \\
w_4 &= w_1 + w_3 + \delta_4, \quad \|\delta_4\| \le \ewed \|w_1 + w_3\| 
=  \ewed \|\m_C z + \delta_1 + \omega_2\delta_2 + \delta_3\|.
\end{align*}
This implies that $w_4 = \m_Cz + \delta_C, \, \delta_C = \delta_1 - \omega_2 \delta_2 + \delta_3 + \delta_4$. The theorem now follows from noting that
\begin{align*}
 \|\delta_C\| &\le \|\delta_1 - \omega_2 \delta_2 + \delta_3\| + \ewed \|\m_C z + \delta_1 + \omega_2\delta_2 + \delta_3\|\\
&\le \|\m_C\| \|z\| \ewed + ( \|\delta_1\| + \|\omega_2\delta_2\| + \|\delta_3\|)(1 + \ewed).
\end{align*}

\begin{rem}
Chebyshev iteration based on $A$ preconditioned by its diagonal $D$ can be formed from two sweeps of damped Jacobi with error propagation operators $I - s_j D^{-1}A = I - M^{(j)}A,$ where $M^{(j)} = M_{ii}^{(j)} = \bigO\left(\frac{\kappa(D)}{\|A\|}\right), \, j=1, 2.$ To estimate $\alpha_{\m_C}$ for this case, we can mimic the proof of Theorem~\ref{thm:Chebyshev}, but with the understanding now that the $\omega_j$ are matrices: $\omega_1 = M^{(2)} + M^{(1)}$ and $\omega_2 = M^{(2)} AM^{(1)}$. If the diagonal matrices $M^{(j)}$ have been formed accurately beforehand, perhaps in the setup phase at higher precision, then the line of reasoning is much the same as the above proof with two extra steps to account for the increased complexity of $\m_C$. The resulting estimate is of the same order as that in Theorem~\ref{thm:Chebyshev} with the exception that a power of $\kappa(D)$ appears in the implied constant. When the diagonal entries of $A$ are widely varying, it may be more effective to construct $D^{-1}A$ in higher precision before it is used in V-cycles in order to avoid an explicit dependence on the condition number of $D$ appearing in the estimate for $\alpha_{\m_C}$.
\end{rem}

\end{proof}

\bibliographystyle{siamplain}
\bibliography{mpmg}